\documentclass{amsart}
\RequirePackage[colorlinks,citecolor=blue,urlcolor=blue,linkcolor=blue]{hyperref}
\hypersetup{
colorlinks = true,
citecolor=blue,
urlcolor=blue,
linkcolor=blue,
pdfpagemode = UseNone
}
\usepackage{graphicx,xspace,colortbl,url}
\usepackage{amsmath,amsthm,amsfonts}
\usepackage{xcolor}
\usepackage{fancybox}
\usepackage{pdfsync}
\usepackage[T1]{fontenc}
\usepackage{enumerate}
\usepackage{geometry}

\def\tofdd{\stackrel{\rm f.d.d.}{\Longrightarrow}}
\def\tofdu{\tofdd}
    \def\qed{\hfill$\sqcap\kern-8.0pt\hbox{$\sqcup$}$\\}
    \def\beq{\begin{eqnarray}}
    \def\eeq{\end{eqnarray}}
    \def\beqq{\begin{eqnarray*}}
    \def\eeqq{\end{eqnarray*}}

    \def\re{\textnormal {Re}}
    \def\im{\textnormal {Im}}
    \def\p{{\mathbb P}}
    \def\e{{\mathbb E}}
    \def\r{{\mathbb R}}

    \def\d{{\textnormal d}}
    \def\i{{\textnormal i}}

    \def\vv{{\textnormal v}}

    \newcommand{\erfc}{{\rm erfc}}
    \def\wt#1{\widetilde{#1}}
\newtheorem{theorem}{Theorem}[section]
\newtheorem{lemma}[theorem]{Lemma}

\theoremstyle{definition}

\newtheorem{remark}[theorem]{Remark}

\def\topp#1{^{(#1)}}

\newcommand{\C}{  \mathsf c }
\newcommand{\A}{  \mathsf a}

\numberwithin{equation}{section}

 \newcommand{\arxiv}[1]{
\setcounter{oldeq}{\value{equation}}
 \addtocounter{usesofarxiv}{1}
 \setcounter{equation}{0}
\def\theoldeq{\theequation}
\def\theequation{x-\arabic{usesofarxiv}.\arabic{equation}}
\def\theequation{\arabic{section}.\arabic{usesofarxiv}.\arabic{equation}}
\def\theequation{\thesection.\arabic{usesofarxiv}.\arabic{equation}}
  \colorlet{shadecolor}{gray!10}
{\footnotesize
\begin{shaded}#1
\end{shaded}
   \setcounter{equation}{\value{oldeq}}
\numberwithin{equation}{section}
}\color{black}}

\newcommand{\la}{\lambda}
\newcommand{\eps}{\varepsilon}

  \usepackage{soul}
   \newcommand{\comment}[1]{}

\def\vv#1{{\boldsymbol #1}}

\newcommand{\R}{\r}

\def\topp#1{^{(#1)}}
   
     \newcommand{\Kab}{\mathfrak{K}_{\A,\C}\topp \tau}

\usepackage{pdfsync}
\usepackage{color,graphicx,pgf,amsmath,float,amssymb}
 \usepackage{framed}
\newcounter{oldeq}
\newcounter{usesofarxiv}

\author{W{\l}odek Bryc}
\address
{
W{\l}odzimierz Bryc\\
Department of Mathematical Sciences\\
University of Cincinnati\\
2815 Commons Way\\
Cincinnati, OH, 45221-0025, USA.
}
\email{wlodek.bryc@gmail.com}

\author{Alexey Kuznetsov}
\address
{
Alexey Kuznetsov\\
Department of Mathematics and Statistics\\
York University, 4700 Keele Street
\\ Toronto, Ontario, M3J 1P3, Canada
}
\email{akuznets@yorku.ca}
\title[Markov limits of steady states  of the KPZ equation]{Markov   limits of steady states  of the KPZ equation on an interval }

\keywords{KPZ fixed point;KPZ steady state;Markov representations}

\subjclass[2020]{60J35;60K40;82C24}

\begin{document}

  \begin{abstract} This paper builds upon the research of Corwin and Knizel  \cite{CorwinKnizel2021} who proved the existence of stationary measures for the KPZ equation on an interval and characterized them through a Laplace transform formula. Ref. \cite{Bryc-Kuznetsov-Wang-Wesolowski-2021}  found a probabilistic description of the stationary measures in terms of a Doob transform of some Markov kernels; essentially at the same time, another description connecting the stationary measures to the exponential functionals of the Brownian motion appeared in  \cite{barraquand2022steady}.

   Our first main result  clarifies and proves the equivalence of the two probabilistic description of these stationary measures.
    We then use the Markovian description to give rigorous proofs of some of the results claimed in \cite{barraquand2022steady}. We analyze how the stationary measures of the KPZ equation on $[0,\tau]$ behave at large scale, as $\tau$ goes to infinity.
   We also analyze the behaviour of the stationary measures of the KPZ equation on $[0,\tau]$ without rescaling, when $\tau$ goes to infinity. %
   Finally, we analyze the measures on $[0,\infty)$ at large scale, which according to \cite{barraquand2022steady} should correspond to stationary measures of  a hypothetical KPZ fixed point on $[0,\infty)$.

   \end{abstract}
      \maketitle

\arxiv{This is an expanded version of the paper with additional material.}

   \section{Notation and background}

   The Kardar-Parisi-Zhang (KPZ)
equation was proposed in \cite{kardar1986dynamic} as a model   for the evolution of the profile of a growing interface  driven by the space-time white noise $\zeta$. The interface profile  is
described by a height   function $H(t,s)$, where $t\geq 0$ is a time variable and $s$ is a spatial variable, which in one spatial dimension
 formally satisfies
\begin{equation}
  \label{KPZ}  \partial_t H(t,s)= \tfrac12 \partial_s^2 H(t,s)+\tfrac12\left(\partial_s H(t,s)\right)^2+\zeta(t,s).
\end{equation}
We refer the reader to the reviews \cite{Corwin2012,quastel2011introduction} about the KPZ equation and its universality class.

 We are interested in \eqref{KPZ} on a
 finite interval,  $s\in[0,1]$, with Neumann boundary conditions
 \begin{equation}
    \label{Nbdry} \partial_s H(t,0)=\C/2,\;  \partial_s H(t,1)=-\A/2.
 \end{equation}
Since the KPZ equation is ill-posed,  the standard notion of the solution is in the Cole-Hopf sense, i.~e. that
 they are defined as $H(t,s)=\log Z(t,s)$, where $Z(t,s)$
 solves the stochastic heat equation
  \begin{equation}
    \label{eq:SHE}
    \partial_t Z(t,s)=\tfrac12 \partial_s^2 Z(t,s)+Z(t,s)\zeta(t,s),\; \; t\geq 0,\; s\in[0,1]
  \end{equation}
  with   boundary conditions \eqref{Nbdry}
  \begin{equation}
    \label{Rbdry}
    \partial_s Z(t,s)|_{s=0}=(\C-1)Z(t,0)/2,\quad  \partial_s Z(t,s)|_{s=1}=(1-\A)Z(t,1)/2.
  \end{equation}
The solutions of \eqref{eq:SHE} are not differentiable, so the meaning of
  \eqref{Rbdry} is not obvious.  A proper definition   was given in \cite[Definition 2.5]{corwin2018open}, see also  \cite{gerencser2019singular,gonccalves2020derivation}.

A stationary measure for the open KPZ equation is the law of a random function    $(\widetilde H_s)_{s\in[0,1]}$
with $\widetilde H_0=0$, defined by the property that if $H(s,t)$ is a Cole-Hopf solution of equation \eqref{KPZ} with boundary values \eqref{Nbdry} and with the initial condition
$H(0,s) = \widetilde H_s$ for  $s \in  [0,1]$, then the law of $(H(t,s)-H(t,0))_{s\in[0,1]}$ does not depend on $t$.
  In a breakthrough paper \cite{CorwinKnizel2021} Corwin and Knizel
    proved existence of such stationary measures  and  determined their multivariate Laplace transform under the restriction $\A+\C\geq 0$.
   Papers \cite{Bryc-Kuznetsov-Wang-Wesolowski-2021} and \cite{barraquand2022steady} inverted this Laplace transform.  Two  %
   different
   representations of the stationary measure obtained in these papers are the starting point of this paper.

     For ease of comparison   we restate   results in \cite{Bryc-Kuznetsov-Wang-Wesolowski-2021} and \cite{barraquand2022steady} in  common parametrization.  To this end, as in
     \cite{Bryc-Kuznetsov-Wang-Wesolowski-2021}
      we replace  two real parameters $u,v$ from \cite{barraquand2022steady,CorwinKnizel2021} by  $\A=2v,\C=2u$. (Boundary parameters in \cite{parekh2019kpz} are thus $A=\C/2$, $B=\A/2$.)  Both papers  find it convenient to replace interval $[0,1]$ with  $[0,\tau]$, $\tau>0$. In our re-write, the length of the interval $\tau$  corresponds to $L$ in \cite{barraquand2022steady}, and it is four times longer than the  length parameter   in \cite{Bryc-Kuznetsov-Wang-Wesolowski-2021}. Since time variable plays no role in this paper, and we are interested in Markov processes in the space variable, following  \cite{Bryc-Kuznetsov-Wang-Wesolowski-2021} we will use variable $t$  as the spatial index of the random fields on $[0,\tau]$.   This index was denoted by $x$ in \cite{barraquand2022steady} and by $X$ in \cite{CorwinKnizel2021}.

  In this notation, Barraquand and Le Doussal \cite{barraquand2022steady} represent the stationary measure of the KPZ equation on an interval as
  \begin{equation}
    \label{X2H} \left(\widetilde H_t\right)_{t\in[0,\tau]}\stackrel{d}{=}\left(B_t+X_t\right)_{t\in[0,\tau]},
  \end{equation}
     where $(B_t)$ is a Brownian motion of variance rate $1/2$  and
  $(X_t)_{t\in[0,\tau]}$   is an independent  stochastic process  with continuous trajectories such that the Radon-Nikodym derivative of its law $ \p_X$ on $C[0,\tau]$ with respect to the law $\p_B$ of Brownian motion $(B_t)_{t\in[0,\tau]}$ with variance rate $1/2$ is
  \begin{equation}\label{eq:Bar-LeD}
   \frac {\d \p_X}{\d \p_B}
 =\frac{1}{ \Kab} e^{-\A \beta_\tau} \left(\int_0^\tau e^{-2 \beta_t}\d t\right)^{-\A/2-\C/2},
 \end{equation} %
where we denoted by $\beta=(\beta_t)\in C[0,\tau]$ the argument of the density function.
Process   $(X_t)_{t\in[0,\tau]}$  depends on parameters $\A,\C,\tau$ but  for now we have suppressed this dependence in our notation.

 Representation \eqref{X2H} is established  in \cite{barraquand2022steady} using  the Laplace transform formula of Corwin and Knizel \cite{CorwinKnizel2021}. The argument is given  for $\A,\C>0$ and $\tau=1$, but it is conjectured that it represents the stationary measure for the KPZ equation on any interval $\tau>0$ with  Neumann boundary  conditions for any real $\A,\C$.

Representation \eqref{X2H} is  also established in Ref. \cite{Bryc-Kuznetsov-Wang-Wesolowski-2021} but process $X$ is described differently and only for $\A+\C>0$. The argument there also relies   on  the Laplace transform formula of Corwin and Knizel \cite{CorwinKnizel2021}  with $\tau=1$, with proof that covers also some negative  values for the  parameters, as long as $\A+\C>0$ and $\min\{\A,\C\}>-2$. The stationary measure of the KPZ equation on an interval is represented there in the form  which,
 after a change of the time scale and sign, compare \cite[(1.8)]{Bryc-Kuznetsov-Wang-Wesolowski-2021}, in present notation is %
  \begin{equation}
    \label{Y2H}\left(\widetilde H_t\right)_{t\in[0,\tau]}\stackrel{d}{=}\left(B_t+Y_t-Y_0\right)_{t\in[0,\tau]},
  \end{equation}
     where $(B_t)$ is an independent Brownian motion of variance 1/2  and
 $(Y_t)_{t\in[0,\tau]} =
   \left(Y_{t}\topp{\A,\C }\right)_{t\in[0,\tau]}$    is an $\r$-valued Markov process   with transition probabilities
   \begin{equation}\label{Y:trans}\p(Y_t\topp{\A,\C}=\d y|Y_s\topp{\A,\C}=x)= %
   \frac{H_t(y)}{H_s(x)}p_{t-s}(x,y)\d y,\; 0\leq s<t\leq \tau
   \end{equation}
 and with initial distribution
\begin{equation}\label{Y0}
 \p(Y_0\topp{\A,\C}=\d x)=\frac{1}{C_{\A,\C}\topp \tau} e^{-\C x}H_0(x)\d x,
\end{equation}
where %
  \begin{equation}\label{pt}
p_t(x,y)= \frac{2}{\pi} \int_0^{\infty} e^{-t u^2/4} K_{\i u}(e^{-x}) K_{\i u}(e^{-y})\frac{\d u}{|\Gamma(\i u)|^2},\quad x,y\in\r,\; t>0,
\end{equation}
is the Yakubovich heat kernel,
\begin{equation}
  \label{C(ac)} C_{\A,\C}\topp \tau:= \int_{\r^2}e^{-\A x-\C y} p_\tau(x,y)\d x  \d y
\end{equation}
is the normalizing constant, and
\begin{equation}
  \label{Ht} H_t(x)=\int_\r e^{-\A y} p_{\tau-t}(x,y)\d y, \quad 0\leq t<\tau \;,
\end{equation}
(with $H_\tau(x):=e^{-\A x}$) is Doob's $h$-transform. In \eqref{pt}, $K_{\i u}(e^{-x})$ is the modified Bessel K function \eqref{BesselK} with positive argument $e^{-x}$ and imaginary index $\i u$.

 Markov process $(Y_t\topp{\A,\C})_{t\in [0,\tau]}$  is well defined for all $\A+\C>0$ and  it is  expected that \eqref{Y2X}  gives stationary solution \eqref{X2H} for all $\tau>0$ and all $\A+\C>0$, so the assumptions $\tau=1$ and $\min\{\A,\C\}>-2$ in \cite[Proposition 1.4]{Bryc-Kuznetsov-Wang-Wesolowski-2021} should not be needed.

\begin{remark}\label{remark_Doob_1}
Markov process $(Y_t\topp{\A,\C})_{t\in [0,\tau]}$ can be described as follows.   Start with a Brownian motion of variance rate $1/2$ and drift $-\A t/2$ and kill it at rate  $\frac{1}{4} e^{-2x}$. Next, condition this killed process to survive until time $\tau$, and we would obtain a Markov process with the same transition probabilities as $(Y_t\topp{\A,\C})_{t\in [0,\tau]}$. One can check that this description is correct by applying Girsanov's theorem and Doob's $h$-transform, see
\cite{Rogers_Williams:1987,Pinsky:1985}.
  We also note that the review  paper \cite[page 7]{Corwin2022} explains  the construction of process $\left(Y_t^{(\A,\C)}\right)_{t\in [0,\tau]}$ emphasizing nice symmetry between parameters $\A$ and $\C$.
\end{remark}

\arxiv{
The above probabilistic description of the process $(Y_t\topp{\A,\C})$ can be justified as follows. The process
$(X^{(1)}_t)$ that has transition probability density $p_t(x,y)$ can be identified as a Brownian motion of rate $1/2$ that is killed at rate $\frac{1}{4} e^{-2x}$ (see Section 3.2 in \cite{Bryc-Kuznetsov-Wang-Wesolowski-2021}). This process can also be described by its Markov generator
 $$
 {\mathcal L}_{X^{(1)}} f(x) = \frac{1}{4} f''(x) - \frac{1}{4} e^{-2x}.
 $$
Fix $\A\in\r$ and consider  the process $(X^{(2)}_t)$, which has transition probability density
  \begin{equation}\label{pt2}
  p^{(2)}_t(x,y)= e^{-\A(y-x) - \frac{\A^2}{4} t} p_t(x,y).
 \end{equation}
 Girsanov's theorem tells us that this modification of the transition density adds a drift of $-\A t/2$ to the underlying Brownian motion. In other words, the process $(X^{(2)}_t)$
 is a Brownian motion of variance rate $1/2$ with drift $-\A t/2$ that is killed at rate $\frac{1}{4} e^{-2x}$.  The Markov generator of the process $(X^{(2)}_t)$ is
 $$
{\mathcal L}_{X^{(2)}} f(x) = \frac{1}{4} f''(x) - \frac{\A}{2}  f'(x) - \frac{1}{4} e^{-2x}.
 $$
 This latter result can also be checked by direct calculation:
 we know that $p_t(x,y)$ satisfies backward Kolmogorov equation
 $$
 \partial_t p=\frac{1}{4} \partial_x^2 p - \frac{1}{4} e^{-2x} p,
 $$
and from here one can deduce (using \eqref{pt2}) that $p^{(2)}_t(x,y)$ satisfies
$$
\partial_t  p^{(2)}=\frac{1}{4} \partial_x^2  p^{(2)} - \frac{\A}{2} \partial_x  p^{(2)} - \frac{1}{4} e^{-2x}  p^{(2)}.
$$
Next, for $0<s<\tau$ we define
$$
H^{(2)}_s(x)=\int_{\R}  p^{(2)}_{\tau-s}(x,y) d y.
$$
It is clear that $H^{(2)}_s(x)$ is the probability that process  $(X^{(2)}_t)$, started from point $x$ at time $s$, survives  up to time  $\tau$. Then the transition probability density of process  $(X^{(2)}_t)$ conditioned to survive up to time $\tau$ is given by
$$
\frac{H^{(2)}_{t}(y)}{H^{(2)}_{s}(x)} p^{(2)}_{t-s}(x,y), \quad 0 < s<t < \tau.$$
Using  \eqref{Ht} and \eqref{pt2} we obtain
$
H_t(x)=e^{-\A x + \frac{\A^2}{4}(\tau-t)} H^{(2)}_{t}(x),
$
and this equation and some simple algebra allow us to conclude that
$$
\frac{H^{(2)}_{t}(y)}{H^{(2)}_{s}(x)}  p^{(2)}_{t-s}(x,y)
=\frac{H_t(y)}{H_s(x)}p_{t-s}(x,y), \quad 0< s<t < \tau.
$$
We thus obtain the transition probabilities \eqref{Y:trans} and we conclude that the process $(Y_t^{(\A,\C)})$ has the transition probability density as the process $(X^{(2)}_t)$ conditioned to survive up to time $\tau$.
}

Formulas \eqref{X2H} and \eqref{Y2H} indicate that the two representations are equivalent, and that process $X$, which is defined for all real  $\A,\C$,  can be represented     as the process of Markov differences \begin{equation}
  \label{Y2X} (X_t)_{t\in[0,\tau]}\stackrel{d}{=}(Y_t\topp{\A,\C}-Y_0\topp{\A,\C})_{t\in[0,\tau]},
\end{equation}
 for $\A+\C>0$ .
 This fact can be read out from  \cite[formula (16)]{barraquand2022steady} when  $\A,\C>0$. We will give an argument that works for the parameters in the entire admissible range $\A+\C>0$   for process $(Y_t\topp{\A,\C})$.
\begin{theorem}
  \label{Thm:Y2K}  For   $\A+\C>0$,
  process $\left(Y_t\topp{\A,\C}-Y_0\topp{\A,\C}\right)_{t\in [0,\tau]}$ has the same law as  process $(X_t)_{t\in [0,\tau]}$ defined by \eqref{eq:Bar-LeD}.
\end{theorem}
The proof  is in Section \ref{Sec:proofY2K}.   In the proof, we   show that  the normalizing constants  in \eqref{eq:Bar-LeD} and   \eqref{C(ac)} are related as follows:
\begin{equation}
  \label{K2C}
  {\mathfrak K}_{\A,\C}^{(\tau)}=\tfrac{2^{1-\A-\C}}{\Gamma(\tfrac{\A+\C}{2})} { C}_{\A,\C}^{(\tau)}.
\end{equation}
To point out the difference between the two representations \eqref{X2H} and \eqref{Y2H}, consider the case $\A+\C=0$. By Cameron-Martin theorem, formula  \eqref{eq:Bar-LeD} implies that  process $(X_t)$ with $\A+\C=0$ has the same law as $(B_t+ \C t/2)$,  so  representation \eqref{X2H}    is  in agreement with  \cite[Theorem 1.2(3)]{CorwinKnizel2021}.
From Theorem \ref{Thm:Y2K}, it follows that if $\eps=\A_\eps+\C_\eps\searrow 0$ while $\C_\eps\to \C$ then  %
$$(Y_t\topp{\A_\eps,\C_\eps}-Y_0\topp{\A_\eps,\C_\eps})_{t\in [0,\tau]} \Rightarrow (B_t+\C t/2)_{t\in[0,\tau]} \mbox{ as } \eps\to 0.$$
However, if $\C\ne 0$ then Markov process
$(Y_t\topp{\A_\eps,\C_\eps})_{t\in[0,\tau]}$ cannot converge because
\begin{equation}
   \label{a+b=0:1}
    \eps Y_{0}\topp{\A_\eps,\C_\eps } \Rightarrow   \gamma_1 \mbox{ as ${\eps\searrow 0}$},
\end{equation}
where $\gamma_1$ is an exponential random variable of mean 1. So   representation \eqref{Y2H} does not extend ``by continuity'' to this case.
(Verification  of \eqref{a+b=0:1}  appears in Section \ref{Sec:P-a+b=0}.)

The goal of this paper is to investigate which of the limits discovered in \cite{barraquand2022steady} for process $X$ extend to process $Y$. When this happens, we can describe the limits   in \cite{barraquand2022steady}  as differences of  explicit Markov processes.
To identify a limiting Markov process, we only need  to prove convergence of finite dimensional distributions, which follows from convergence of the initial laws and transition probabilities,  which is a simpler task than  deducing weak convergence  from  the convergence of the Feller semigroups as in \cite[ Ch. 4, Thm. 2.5]{ethier2009markov}, and seems to be also easier than the approach based on Radon-Nikodym derivatives as in \cite{barraquand2022steady,hariya2004limiting}.
 \subsection*{Notation} Throughout  the paper, $(B_t)$ denotes Brownian motion of variance $1/2$, i.e., process $(B_t)_{t\geq 0}$ has the same law as process  $(W_{t/2})_{t\geq 0}$ or process $(\tfrac{1}{\sqrt{2}}W_t)_{t\geq 0}$, where $(W_t)_{t\geq 0}$ is the standard Wiener process. By $\gamma_p$ we denote a Gamma random variable with shape parameter $p>0$ and scale parameter one, i.e.
 $$
 {\mathbb P}(\gamma_p \in \d x)=\frac{1}{\Gamma(p)} x^{p-1}e^{-x}{\bf 1}_{\{x>0\}} \d x.$$
 In particular, $\gamma_1$ is the standard exponential random variable of mean 1. By $\Rightarrow$ we denote weak convergence (of the corresponding measures on $\r$ or on $C[0,1]$) and by
 $\tofdd$ we denote the convergence of finite dimensional distributions. For argument  $x>0$ and complex  index $z$ we denote by
\begin{equation}
  \label{BesselK}
  K_z(x)=\int_0^\infty e^{-x \cosh w}\cosh(z w)\d w
\end{equation}
 the modified Bessel K function.

\section{Asymptotic results} We first describe  the limits derived in \cite{barraquand2022steady} for process $X$ and then in Section \ref{Sec:Lims} we discuss the corresponding limits for Markov process $Y$.
 In view of Theorem \ref{Thm:Y2K}, limits for process $Y$ give the corresponding limits for process $X$, but we shall see that the converse sometimes fails.

 To indicate how Barraquand and Le Doussal   scale the  parameters for their limits,  we need to write them explicitly, so in Section \ref{sec:descr} we shall write  $\left(X\topp{\A,\C}_t\right)_{t\in[0,\tau]}$  for the stochastic process with density \eqref{eq:Bar-LeD}.

 \subsection{Description of limits in \cite{barraquand2022steady}}\label{sec:descr}
 Barraquand and Le Doussal  \cite{barraquand2022steady} determine several interesting limits.
 Some of these limits are expected to correspond to the stationary measures of a hypothetical KPZ fixed points on an interval and on the half-line. This is the area of intense current research, and some new preprints appeared while this paper was undergoing a
 revision.   At this time however, we note that the  KPZ fixed point was defined rigorously in \cite{matetski2016kpz} only on $\r$.

 The following describes the limits determined in \cite{barraquand2022steady}.
 \begin{itemize}
       \item[(i)]\label{KPZ:fixed}
       It is claimed in \cite{barraquand2022steady} that the large scale limit $\lim_{\tau\to \infty} \left(\frac{1}{\sqrt{\tau}}X\topp{\A/\sqrt{\tau},\C/\sqrt{\tau}}_{t \tau}\right)_{t\in[0,1]} $ should give the {\em stationary measures of the hypothetical KPZ fixed point on an interval}.         The limit is described in  \cite{barraquand2022steady} by   density \eqref{B:eta:dens} and also by the heat kernel and the Laplace transform  \cite[supplement, formulas (51), (54)]{barraquand2022steady}.

       By universality, such limits should correspond to stationary measures of a hypothetical KPZ fixed point on an interval, and indeed a recent result in \cite{Bryc-Wang-Wesolowski-2022} derives the same limit from an open Asymmetric Simple Exclusion Process.

At this time, the notion of the KPZ fixed point on an interval has not yet been defined.

\vspace{0.25cm}
           \item[(ii)]\label{KPZ:halfline} It is conjectured in \cite{barraquand2022steady} that  {\em the  stationary solution of the half-line KPZ equation} is given by  $\lim_{\tau\to\infty}\left(X\topp{\A,\C}_t\right)_{t\in[0,\tau]}$, a limit which was already studied in \cite[Theorem 1.3]{hariya2004limiting} for different reasons.

        The  form of the limit depends on  the parameters  $\A,\C$ as follows:

           \vspace{0.15cm}

               \begin{itemize}
                 \item[(a)] \label{KPZ-halfline:mc} If $\A\geq0, \C\geq 0$ then the limit $(\widetilde Z_t\topp{\C})_{t\geq 0}$ does not depend on $\A$, and is expressed by    \eqref{KPZ:half:hd} with $\A=0$.
                   \item[(b)]
                   \label{KPZ-halfline:hd} If $\A\leq0, \C\geq \A$ then the limit $(\widetilde Z_t\topp{\A,\C})_{t\geq 0}$ is given by  \eqref{KPZ:half:hd}  %
                    (with $\A$ replaced by $-\A$)  and depends on both $\A, \C$.
                    \item[(c)]  \label{KPZ-halfline:ld} If $\A\geq \C, \C\leq 0$ then the limit is $(B_t+t \C/2)$   and does not depend on $\A$.
               \end{itemize}

               \vspace{0.25cm}

                According to a recent preprint \cite{Barraquand2022},
            these limits are indeed the stationary solutions for the KPZ equation on
            the half-line.

            \vspace{0.25cm}

            \item[(iii)] \label{KPZ:LS}
             It is claimed in \cite{barraquand2022steady} that the large scale limit  as $T\to \infty$  of  the stationary measures  on the half-line from point (ii) should give the stationary measures of the
             {\em KPZ fixed point on the half-line}. %
             These limits are:

        \vspace{0.15cm}
\begin{itemize}

   \item[(a)]  \label{KPZ:LS:mc} If $\A>0,\C>0$ then
    \begin{equation}
   \label{eq:BD2136}
   \frac{1}{\sqrt{T}}\left(\widetilde Z_{tT}\topp {\C/\sqrt{T}}\right)_{t\geq 0} \Rightarrow
   \left(B_t+\max\left\{0, -\tfrac{2}{\C}\gamma_{1}-2\min_{0\leq s\leq t}B_s\right\}\right)_{t\geq 0}  \mbox{ as $T\to\infty$},
 \end{equation}
 where  $\gamma_1$ is an %
  independent
  exponential random variable.

    \item[(b)]  \label{KPZ:LS:hd} If $\A<0$ and $\C>\A$, then
        \begin{multline}
          \label{eq:BD2137}
            \frac{1}{\sqrt{T}}\left(\widetilde Z_{tT}\topp {\A/\sqrt{T},\C/\sqrt{T}}\right)_{t\geq 0} \\
            \Rightarrow
   \left(B_t+\A t/2+\max\left\{0, -\tfrac{2}{\C-\A}\gamma_{1}-2\min_{0\leq s\leq t}\{B_s+\A s/2\}\right\}\right)_{t\geq 0}  \mbox{ as $T\to\infty$},
        \end{multline} %
     where  $\gamma_1$ is an independent exponential random variable.
        \item[(c)]  \label{KPZ:LS:ld} For $\A\geq \C, \C\leq 0$, process $(B_t+t \C/2)_{t\geq 0}$ is invariant under the scaling used in the limit so it should be the stationary measure of the hypothetical KPZ fixed point on the half-line.
\end{itemize}

      \vspace{0.25cm}

At this time, the notion of the KPZ fixed point on the half-line has not yet been defined.

   \vspace{0.25cm}

 \item[(iv)]\label{KPZ:EW}   %
 For completeness, we also mention that according to \cite{barraquand2022steady} , the small scale limit  $$\lim_{\tau\to 0} \left(\frac{1}{\sqrt{\tau}}X\topp{\A/\sqrt{\tau},\C/\sqrt{\tau}}_{t \tau}\right)_{t\in[0,1]} $$ is described in   \eqref{BD:lim:EW}. We learned from the review that this limit should correspond to the stationary measure of
 {\em the  Edwards-Wilkinson equation} that is the stochastic equation with additive noise
 $$
 \partial_t Z=\frac12\partial_{xx}Z+\zeta.
 $$
(Our techniques do not allow us to say anything about this limit.)
 \end{itemize}

 \medskip
We shall represent some of these limits in a different form using  the corresponding limits of the Markov process $\left(Y_{t}\topp{\A,\C }\right)_{t\in[0,\tau]}$.

\subsection{Limits of Markov process $Y$}\label{Sec:Lims}

\subsubsection{Markov process for the stationary measure  of the (hypothetical) KPZ fixed point on an interval}
The following result is a version of point (i) in Section \ref{sec:descr}.
 It is related to \cite{Bryc-Wang-Wesolowski-2022}, who determined that up to an irrelevant scaling factor, the fluctuations of particle density for ASEP with varying parameters for fixed $q$ are given by the same Markov process. We use the notation $(\widetilde \eta_t)$ and $(\eta_t)=(\widetilde \eta_t-\widetilde \eta_0)$ to connect with that paper. (Strictly speaking, the process denoted by $\eta$ in \cite{Bryc-Wang-Wesolowski-2022} is $\sqrt{2}\eta$ in our notation. But the actual process appearing in the limit  \cite[Theorem 1.5]{Bryc-Wang-Wesolowski-2022} is $\eta$.)
  \begin{theorem}%
  \label{Thm:KPZ:fixed}
  For fixed $\A+\C>0$ we have %
    \begin{equation*}\label{BD:lim}
  \left(\frac{1}{\sqrt{\tau}}Y_{t \tau}\topp {\A/\sqrt{\tau},\C/\sqrt{\tau}}\right)_{t\in[0,1]}\tofdu \left(\widetilde \eta_t\topp{\A,\C}\right)_{t\in[0,1]} \mbox{  as $\tau\to\infty$},
\end{equation*}
   where $(\widetilde \eta_t)_{t\in[0,1]}$ is a Markov process with the initial distribution
      \begin{equation}
  \label{eta0}
  \p(\widetilde \eta_0=\d x)=\frac{1}{\mathfrak C_{\A,\C}}e^{-\C x} h_0(x) {\mathbf 1}_{\{x>0\}}\d x.
\end{equation}
      and transition probabilities
  \begin{equation}\label{eta:trans}
  p_{s,t}(x,\d y)=\frac{h_t(y)}{h_s(x)}g_{t-s}(x,y)\d y, \quad 0\leq s<t\leq 1,
  \end{equation}
      where
\begin{equation}
  \label{eta:heat}g_t(x,y)=\frac{2}{\pi}\int_0^\infty e^{-t u^2/4}\sin(x u)\sin(y u)\d u = \frac{1}{\sqrt{\pi t}} \Big(e^{-\frac{(x-y)^2}{ t}}-e^{-\frac{(x+y)^2}{ t}} \Big),\; x,y,t>0,
\end{equation}
the normalizing constant is
\begin{equation}
  \label{C:norm:eta}
  \mathfrak C_{\A,\C}=\int_{\r_+^2} e^{-\C x - \A y}g_1(x,y)\d x\d y
\end{equation}
and  Doob's $h$-transform is
$$h_t(x)= \int_{\r_+}g_{1-t}(x,y)e^{-\A y} \d y, \;x>0, \; 0\leq t<1 $$
with $h_1(x):=e^{-\A x}{\mathbf 1}_{\{x>0\}}$. %
\end{theorem}
The proof of this result is in Section \ref{Sect:KPZ:fixed}.

\begin{remark}\label{remark_Doob_2}
Function $g_t(x,y)$ defined in
\eqref{eta:heat} is the transition probability density function of a Brownian motion of variance rate $1/2$ killed at the time when it first hits zero. Applying Girsanov's theorem and Doob's $h$-transform we conclude that $p_{s,t}(x,\d y)$ defined in
\eqref{eta:trans} is the transition probability kernel of Brownian motion of variance rate $1/2$ and drift $-\A t/2$ and conditioned to stay positive for $0\le t \le 1$.
\end{remark}

\arxiv{
It is a well-known fact that the function $g_t(x,y)$ in \eqref{eta:heat} is the transition probability density of Brownian motion of variance rate $1/2$ that is killed when it first hits zero. Define  %
\begin{equation*}\label{G-Girsanov}
 p_t^{(3)}(x,y)= e^{-\A(y-x) - \frac{\A^2}{4} t} g_t(x,y).
\end{equation*}
By Girsanov's theorem, $p^{(3)}_t(x,y)$ is the transition probability density function of the process $(X^{(3)}_t)$, which is the Brownian motion of variance rate $1/2$ and drift $-\A t/2$ that is killed when it first hits zero.
Then the function
$$
H^{(3)}_s(x)=\int_0^{\infty}
 p_{1-s}^{(3)}(x,y) d y, \;\;\; x>0, \; 0<s<1,
$$
is the probability that the process $(X^{(3)}_t)$ started at point $x$ at time $s$ stays positive until time $1$. Let $(X^{(4)}_t)_{t\in [0,1]}$ be the process $(X^{(3)}_t)$ conditioned to survive until time $1$ (equivalently, we could condition on the event that the process $(X^{(3)}_t)$ stays positive for all $t\in [0,1]$). By Markov property it is easy to see that the transition probability density of the process $(X^{(4)}_t)_{t\in [0,1]}$ is
$$
\frac{H^{(3)}_{t}(y)}{H^{(3)}_{s}(x)} p^{(3)}_{t-s}(x,y) \d y, \;\;\; x,y>0, \; 0<s<t<1.
$$
After some simple algebra one can check that this is the same transition probability density function as $p_{s,t}(x,y)$ given in
\eqref{eta:trans}, thus the process
$(\tilde \eta_t)$ has the same transition probabilities as the Brownian motion of variance rate $1/2$ and drift $-\A t/2$ that is conditioned to stay positive for all $t\in [0,1]$.
}

We remark that  according to \cite[formula  (31)]{barraquand2022steady}, for process $(\eta_t)_{t\in [0,1]}:=(\widetilde \eta_t - \widetilde \eta_0)_{t\in [0,1]}$,    we have %
\begin{equation}
  \label{B:eta:dens}
    \frac{\d \p_\eta}{\d \p_B}=\frac{1}{(\A+\C)\mathfrak C_{\A,\C}} e^{-\A \beta_1+(\A+\C) \min_{0\leq t\le 1} \beta_t}
\end{equation}
with  the normalizing constant    given by \eqref{C:norm:eta}.
This can be verified by combining Theorem \ref{Thm:Y2K} and Theorem \ref{Thm:KPZ:fixed}.

\arxiv{
Indeed, these two theorems show that for  $\A+\C>0$, the law of the process $\eta\topp{\A,\C}$ is the limit as $\tau\to\infty$ of the process
$$\left(X_t\topp \tau\right)_{t\in[0,1]}=\left(\frac{1}{\sqrt{\tau}}X_{ t \tau }\topp{\A/\sqrt{\tau},\C/\sqrt{\tau}}\right)_{t\in[0,1]},$$
where $X:=\left(X_t\topp{\A,\C}\right)_{t\in[0,\tau]}$ is the process introduced in \eqref{Y2X} with the Radon--Nikodym derivative on  $C[0,\tau]$ with respect to the law of $B\topp \tau:=(B_t)_{t\in [0,\tau]}$   given by \eqref{eq:Bar-LeD}.
 Since process $B\topp \tau$ has the same law as process $\left(\sqrt{\tau}B\topp 1 _{t/\tau}\right)_{t\in[0,\tau]}$, we have

\begin{equation}
    \label{RN-X-tau}
    \frac{d
\p_{ X\topp \tau}
}{d \mathbb{P}_{B\topp 1}}(\wt \beta)
=\frac{1}{\mathfrak K_{\A/\sqrt{\tau},\C/\sqrt{\tau}}\topp \tau}e^{-\A \wt \beta_1}\left( \tau \int_0^1 e^{-2 \sqrt{\tau}\wt \beta_t}\d t\right)^{-(\A+\C)/(2\sqrt{\tau})}, \; \wt \beta=(\wt \beta_t)\in C[0,1].
\end{equation}
Indeed, if       $ \frac{d
\p_{ X}
}{d \mathbb{P}_{B\topp \tau}}(\beta)=\Phi\left((\beta_x)_{x\in[0,\tau]}\right)$ for some
 $\Phi:C[0,\tau]\to\r_+$,  then
 $$ \frac{d
\p_{ X\topp \tau}
}{d \mathbb{P}_{B\topp 1}}(\wt \beta)=\Phi\left((\sqrt{\tau}\wt \beta_{t/\tau})_{t\in[0,\tau]}\right)$$ for $\wt \beta\in C[0,1]$.
With
$$\Phi(\beta)=\frac{1}{\mathfrak K}  e^{-\A \beta_\tau/\sqrt{\tau}} \left(\int_0^\tau e^{-2 \beta_t}\d t\right)^{-(\A+\C)/(2\sqrt{\tau})}$$ we get
\begin{multline*}
  \Phi\left((\sqrt{\tau}\wt \beta_{t/\tau})_{t\in[0,\tau]}\right)
=\frac{1}{\mathfrak K}  e^{-\A (\sqrt{\tau} \wt \beta_1)/\sqrt{\tau}} \left(\int_0^\tau e^{-2 \sqrt{\tau}\wt \beta_{t/\tau}}\d t\right)^{-(\A+\C)/(2\sqrt{\tau})}
\\=\frac{1}{\mathfrak K}e^{-\A   \wt \beta_1 } \left(\tau\int_0^1 e^{-2 \sqrt{\tau}\wt \beta_{t}}\d t\right)^{-(\A+\C)/(2\sqrt{\tau})}
\end{multline*}

We now check that the density \eqref{RN-X-tau} converges to the correct limit.  We first verify convergence of the normalization constant. From
\eqref{K2C} and Lemma \ref{Llem:C2}  we get
\begin{multline*}
   \lim_{\tau\to\infty} \mathfrak K_{\A/\sqrt{\tau},\C/\sqrt{\tau}}\topp \tau=
   \lim_{\tau\to\infty}
\tfrac{2 { C}_{\A,\C}^{(\tau)}}{2^{(\A+\C)/(\sqrt{\tau})}\Gamma(\tfrac{\A+\C}{2\sqrt{\tau}})}
 =  2 \lim_{\tau\to\infty}
\tfrac{{ C}_{\A,\C}^{(\tau)}}{\sqrt{\tau}}\lim_{\tau\to\infty} \frac{\sqrt{\tau}}{\Gamma(\tfrac{\A+\C}{2\sqrt{\tau}})}
\\= 2 \mathfrak C_{\A,\C} \frac{\A+\C}{2} =(\A+\C)\mathfrak C_{\A,\C}.
\end{multline*}
 Since $\A+\C>0$,
$$\lim_{\tau\to\infty} \left(\tau \int_0^1 e^{-2 \sqrt{\tau}\beta_t}\d t\right)^{-(\A+\C)/(2\sqrt{\tau})}
=e^{(\A+\C)\min_{0\leq t\leq 1}\beta_t},$$
thus the Radon--Nikodym density \eqref{RN-X-tau} converges to the expression \eqref{B:eta:dens}.
By Scheff\'e's theorem, pointwise convergence of the probability density functions implies weak convergence of measures, thus identifying the law of the limiting process uniquely.

}
\subsubsection{Markov processes for the stationary measure of KPZ on the half-line}
The following is a version of  point
ii(a) in Section \ref{sec:descr}.
 In view of Theorem \ref{Thm:Y2K}, it also provides a different description of the limit process in \cite[Theorem 1.3.(i)]{hariya2004limiting}, representing it as $(Z_t\topp\C-Z_0\topp\C)_{t\geq 0}$.

\begin{theorem}%
\label{T:halfL} %
   If $\A,\C>0$  then %
   \begin{equation*}\label{Half:line:Z}
     \left(Y_{t}\topp{\A,\C }\right)_{t\in[0,\tau]} \tofdu (Z_t\topp \C)_{t\geq 0}  \mbox{  as $\tau\to\infty$},
   \end{equation*}
  where $(Z_t\topp \C)$ is an $\r$-valued Markov process with the initial distribution
  \begin{equation}\label{Z0}
     \p (Z_0\topp \C=\d x)=\frac{4 e^{-\C x} K_0(e^{-x})}{2^\C \Gamma(\C/2)^2}\d x
  \end{equation}
  and with transition probabilities
  \begin{equation}\label{Z:trans}
   \p (Z_t\topp \C=\d y|Z_s\topp \C=x)=\frac{K_0(e^{-y})}{K_0(e^{-x})}p_{t-s}(x,y)\d y,\; 0\leq s<t,
  \end{equation}
where $p_{t}(x,y)$ is given by \eqref{pt}.
  \end{theorem}
  The proof of this result is in Section \ref{Sect:T:halfL}.  By comparing the Laplace transforms (see \eqref{E6.8(26)}) we notice that
  the law  \eqref{Z0} is
  the law of  $-  \log \left( 2 \sqrt{\gamma_{\C/2}\tilde\gamma_{\C/2}}\right)$,  where $\gamma_{\C/2}$ and $\tilde \gamma_{\C/2}$  are independent  Gamma random variables with shape parameter $\C/2$. %

From Theorem \ref{Thm:Y2K}, and \cite{hariya2004limiting} it follows that if $\A+\C>0$ with $\A<0$, then process
$\big(Y_{t}\topp{\A,\C }-Y_{0}\topp{\A,\C }\big)_{t\in[0,\tau]}$  converges in distribution as $\tau\to\infty$.  This suggests the following  result related to point ii(b)  in Section \ref{sec:descr}.

\begin{theorem}
\label{T:halfL:hd}   If $\A+\C>0$, with
$\A\leq 0$ %
then
$$ \left(Y_{t}\topp{\A,\C }\right)_{t\in[0,\tau]}  \tofdd (Z_t\topp{\A,\C})_{t\geq 0}  \mbox{ as $\tau\to\infty$ }$$ where $(Z_t\topp{\A,\C})_{t\geq 0}$
 is an $\r$-valued Markov process with the initial distribution
  \begin{equation}\label{Z0:hd}
     \p (Z_0\topp {\A,\C}=\d x)=\frac{4}{  2^\C \Gamma \left(\frac{\C-\A}{2}\right) \Gamma
   \left(\frac{\A+\C}{2}\right)}e^{-\C x} K_{\A}\left( e^{-x}\right)\d x
  \end{equation}
  and with transition probabilities
  \begin{equation*}\label{Z:trans:hd}
   \p (Z_t\topp {\A,\C}=\d y|Z_s\topp {\A,\C}=x)=\frac{e^{-\A^2t/4}K_\A(e^{-y})}{e^{-\A^2s/4}K_\A(e^{-x})}p_{t-s}(x,y) \d y,\; 0\leq s<t.
  \end{equation*}
\end{theorem}
The proof is in Section \ref{Sec:halfL:hd}. We note that
Markov process $(Z_t\topp{\A,\C})_{t\geq 0}$ is well defined
 when $\vert\A\vert<\C$,
and it is clear that process  $(Z_t\topp \C)_{t\geq 0}$ from Theorem \ref{T:halfL} has the same law as $(Z_t\topp{0,\C})_{t\geq 0}$.  From \eqref{E6.8(26)}  it is easy to use the Laplace transform
to identify the law \eqref{Z0:hd} as the law of   $-\log\left(2 \sqrt{\gamma_{\tfrac{\A+\C}{2}} \tilde \gamma_{\tfrac{\C-\A}{2}}}\right)$ for a pair of independent Gamma random variables.
In light of Remark \ref{remark_Doob_1} the
 transition probabilities
of the two processes identified in Theorems
 \ref{T:halfL} and
\ref{T:halfL:hd} can be described as the Brownian motion of variance rate $1/2$ and drift $-\A t/2$, killed at rate $\frac{1}{4} e^{-2x}$ and conditioned to survive for all $t\ge 0$. These asymptotic results should be compared with Theorem 1.2 in \cite{hariya2004limiting}.

Next, we confirm that
Theorems \ref{T:halfL} and \ref{T:halfL:hd}  give Markov representations for the processes \cite[Theorem 1.3(i)+(iii)]{hariya2004limiting} when $\A+\C>0$, i.e., when their parameter $m>0$. (\cite[Section 5 of Supplement]{barraquand2022steady} shows how to relate the parametrizations.)
Of course, this result follows from \cite[Theorem 1.3(i)+(iii)]{hariya2004limiting} and our previous results. However we include a somewhat more direct proof for completeness.
    \begin{theorem}\label{T:Alexey2}
  If %
  $\A+\C>0$, $|\A|<\C$,
  then  process $(Z_t\topp{\A,\C}-Z_0\topp{\A,\C})_{t\geq 0}$ has the same distribution as   process
   \begin{equation}\label{KPZ:half:hd}
    \left(B_t-\A t/2+\log\big(1+\gamma_{\tfrac{\A+\C}{2}} \int_0^t e^{-2 B_s+\A s} ds \big) \right)_{t\geq 0},
   \end{equation}
  where $\gamma_{\tfrac{\A+\C}{2}}$ is an independent gamma random  variable of shape parameter  $\tfrac{\A+\C}{2}$ and  $(B_t)_{t\ge 0}$ is a Brownian motion of variance $1/2$.
  \end{theorem}

{%
Note that both the initial distribution and the transition probability density function of the process
$(Z_t\topp{\A,\C})_{t\geq 0}$ (as described in Theorem \ref{T:halfL:hd})
is invariant with respect to change $\A \mapsto -\A$, therefore, the same is true for the process given in \eqref{KPZ:half:hd}. This latter result also follows from  \cite{hariya2004limiting}.
The proof of Theorem \ref{T:halfL:hd} is in Section \ref{Sec:Alexey2}.
}

We conclude with a remark about the limit in point ii(c) in Section \ref{sec:descr}.
From Theorem \ref{Thm:Y2K}, and \cite[Theorem 1.3.(ii)]{hariya2004limiting} it follows that if $\A+\C>0$, with $\C\leq 0$ then
 \begin{equation*}\label{Half:line:Z:II}
     \left(Y_{t}\topp{\A,\C }-Y_{0}\topp{\A,\C }\right)_{t\in[0,\tau]} \tofdd (B_t+\C t/2)_{t\geq 0}  \mbox{  as $\tau\to\infty$},
   \end{equation*}
   where $(B_t)_{t\ge 0}$ is the Brownian motion of variance  rate $1/2$.
   However,   when $\C<0$  process $(Y_t\topp{\A,\C})_{t\in[0,\tau]}$  does not converge as $\tau\to\infty$
because under a different scaling
$ Y_{0}\topp{\A,\C }$, which depends on $\tau$ through \eqref{Y0},  is asymptotically normal:
\begin{equation}
   \label{Half:line:N}
   \sqrt{\frac{2}{\tau}}\left( Y_{0}\topp{\A,\C }+\tau \C/2\right)\Rightarrow N(0,1) \;\mbox{  as $\tau\to\infty$}.
\end{equation}

The proof of \eqref{Half:line:N} is in Section \ref{Sec:HalfZ:II}.

\subsubsection{Markov processes for the stationary measures  of the (hypothetical) KPZ fixed point  on the half-line}
 The following is related to point 
 iii(a) in Section \ref{sec:descr}.
  \begin{theorem} \label{T:Z:fixed}  Fix $\C>0$ and  denote  by $(Z_t\topp {\C})_{t\geq 0}$ the Markov process from Theorem \ref{T:halfL}. Then
$$\frac{1}{\sqrt{T}}\left(Z_{tT}\topp {\C/\sqrt{T}}\right)_{t\geq 0} \tofdu ( \rho_t\topp \C)_{t\geq 0} \mbox{  as 
$T\to\infty$}, $$
    where $(\rho_t\topp \C)_{t\geq 0}$ is $\tfrac{1}{\sqrt{2}}$ multiple of
    the 3-dimensional Bessel process  %
    with the  initial distribution
    \begin{equation}
      \label{KPZ:fixed:Z0}
      \p( \rho_0\topp \C=\d x)=\C^2 x e^{-\C x} {\bf 1}_{\{x>0\}} \d x.
    \end{equation}

  \end{theorem}
(Of course, \eqref{KPZ:fixed:Z0} is the gamma law with shape parameter two and scale parameter $\C$, which is the sum of two independent exponential variables, $(\gamma_1+\wt \gamma_1)/\C$.)   The proof is in Section \ref{Sec:T:Z:fixed}.

  \begin{remark}\label{R:36}
  According to \cite[(36)]{barraquand2022steady}, %
   process $\left(\rho_t\topp \C-\rho_0\topp\C\right)_{t\geq 0}$  has the same law as the right hand side of \eqref{eq:BD2136}   and this reference  also noted the connection to the Bessel 3 process. %

   \end{remark}

The following is related to point iii(b) in  Section \ref{sec:descr}.
\begin{theorem} \label{T:half:hd:fixed} %
Let $\A+\C>0$, $\A<0$,
and denote by $(Z_t\topp{\A,\C})_{t\geq 0}$ the Markov process from Theorem \ref{T:halfL:hd}. Then
\begin{equation*}\label{Z2rho}
 \frac{1}{\sqrt{T}}\left(Z_{tT}\topp {\A/\sqrt{T},\C/\sqrt{T}}\right)_{t\geq 0} \tofdu (\rho_t\topp{\A, \C})_{t\geq 0} \mbox{ as } 
 T\to \infty,
\end{equation*}
    where $(\rho_t\topp {\A,\C})_{t\geq 0}$ is
a Markov process with transition probabilities
\begin{equation}
  \label{KPZ:fixed:rho-trans}
  \p(\rho_t=\d y|\rho_s=x)=\frac{e^{-\A^2 t/4} \sinh(\A y)}{e^{-\A^2 s/4} \sinh(\A x)}g_{t-s}(x,y)\d y, \; x>0, y>0, \; 0\leq s<t,
\end{equation}
 where $g_t(x,y)$ is given by \eqref{eta:heat},
    with the initial distribution
    \begin{equation*}
      \label{KPZ:fixed:rho0}
      \p(\rho_0\topp {\A,\C}=\d x)=\mu_{\A,\C}(\d x),
    \end{equation*}
where $\mu_{\A,\C}(\d x)= \frac{c^2-a^2}{2 a} \Big( e^{(a-c)x}-e^{-(a+c)x} \Big){\mathbf 1}_{\{x>0\}}\d x
$
is  the law of the linear combination
$\frac{1}{\A+\C}\gamma_1+\frac1{\C-\A}\tilde \gamma_1$ for a pair of independent exponential random variables.
\end{theorem}
The proof is in Section \ref{Sec:half:hd:fixed}.       The transition probability density  function \eqref{KPZ:fixed:rho-trans} is the same as for the process $B_t - \frac{\A}{2} t$, conditioned to stay positive for all $t\ge 0$, see \cite{Rogers_Pitman:1981}.

\begin{remark}\label{R:37}
  According to \cite[(37)]{barraquand2022steady}, when  $\A<0$ and $\A+ \C>0$,  process $(\rho_t\topp {\A,\C}-\rho_0\topp {\A,\C})_{t\geq 0}$  has the same law as the right hand side of \eqref{eq:BD2137}.
   \end{remark}

\subsubsection{The  Edwards-Wilkinson  limit }
We conclude this section with a comment on the limit from point (iv) in Section \ref{sec:descr}.
 According to \cite[(29)]{barraquand2022steady}, when  $\A+\C>0$,   Theorem \ref{Thm:Y2K} implies that
    \begin{equation}\label{BD:lim:EW}
  \frac{1}{\sqrt{\tau}}\left(Y_{t \tau}\topp {\A/\sqrt{\tau},\C/\sqrt{\tau}}-Y_{0}\topp {\A/\sqrt{\tau},\C/\sqrt{\tau}}\right)_{t\in[0,1]}
  \tofdd \left(B_t+\C t/2-(\A+\C)\frac{t^2}{4}\right)_{t\in[0,1]} \mbox{ as } \tau\to0.
\end{equation}
We expect that process $ \frac{1}{\sqrt{\tau}}\left(Y_{t \tau}\topp {\A/\sqrt{\tau},\C/\sqrt{\tau}}\right)$ does not converge as $\tau\to 0$.

\section{Proof of Theorem \ref{Thm:Y2K}}\label{Sec:proofY2K}
For the proof we choose   a probability space $(\Omega,{\mathcal F}_t, {\mathbb P})$, where $\Omega$ is a space of continuous functions on $[0,\tau]$ and $\p$ is the law of Brownian motion $(B_t)\cong(W_t)/\sqrt{2}$ with variance $1/2$ which was denoted by $\p_B$ in \eqref{eq:Bar-LeD}. Under the measure ${\mathbb P}$ the coordinate process $X_t(\omega)=\omega(t)$ is a Brownian motion with variance $1/2$. We define the exponential functional as $I_t:=\int_0^t e^{-2 X_s} \d s$ and consider the two-dimensional process $(X,I)$. This process is a time-homogeneous diffusion
and from \cite[Theorem 4.1]{MatsumotoYor2005I} (see also \cite[Proposition 2]{yor1992some} or  \cite[formula (2.6)]{donati2002law})  we know that its transitional probability density function  is given by
$$
\mu^{\mathbb P}_t(x_0,z_0;x_1,z_1)=\exp\Big(- \tfrac{e^{-2x_0}+e^{-2x_1}}{z_1-z_0}\Big) \theta\Big(2\tfrac{e^{-x_0-x_1}}{z_1-z_0},\tfrac{t}{2}\Big) \frac{1}{z_1-z_0},
$$
where $t>0$, $x_i\in \r$ and $z_1>z_0\ge 0$. Here $\theta(r,t)$ is the
 (unnormalized) Hartman-Watson density function, which is   succinctly defined in \cite{hartman1974normal}
by   its Laplace transform
\begin{equation}\label{Def-HW}
  \int_0^\infty e^{-\la^2 t}\theta(\xi,t)\d t=\mathbf{I}_\la(\xi), \;\;\; \la>0, \; \xi>0,
\end{equation}
where ${\mathbf I}_\la(\xi)$ is the modified Bessel function. For literature on properties of the Hartman-Watson density we refer to \cite {MatsumotoYor2005I}, \cite{Wisniewolski2020}, and the references therein.

For $t\in [0,\tau]$, $x\in \r$ and $z\ge 0$ we define
$$
\phi_t(x,z)=\e^{\mathbb P} \Big[ e^{-\A X_t} I_t^{-\A/2-\C/2} \vert X_0=x, I_0=z\Big].
$$
For $\tau>0$ and $\A+\C>0$ we define
$$
Z_{\tau}:=\frac{1}{{\mathfrak K}_{\A,\C}^{(\tau)}} e^{-\A X_{\tau}} I_{\tau}^{-\A/2-\C/2},
$$
where ${\mathfrak K}_{\A,\C}^{(\tau)}=\phi_{\tau}(0,0)$ is a normalizing constant which ensures  ${\mathbb E}^{\mathbb P}[Z_{\tau}|X_0=I_0=0]=1$. For $t\in [0,\tau]$ we define a martingale $Z_t:=\e^{\mathbb P}[Z_{\tau}\vert {\mathcal F}_t]$ and compute
\begin{align*}
Z_t&=\e^{\mathbb P}[Z_{\tau}\vert {\mathcal F}_t]=
\frac{1}{{\mathfrak K}_{\A,\C}^{(\tau)}} \e^{\mathbb P}\Big[ e^{-\A X_{\tau}} I_{\tau}^{-\A/2-\C/2} \Big\vert {\mathcal F}_t \Big]
=\frac{1}{{\mathfrak K}_{\A,\C}^{(\tau)}}
\phi_{\tau-t}(X_t,I_t).
\end{align*}
We define a change of measure $\d {\mathbb Q}_t/\d {\mathbb P}_t=Z_t$ (where ${\mathbb P}_t$ is the restriction of ${\mathbb P}$ on ${\mathcal F}_t$ and similarly for ${\mathbb Q}_t$). Take an arbitrary bounded measurable function $f(x,z)$ and compute for $0<s<t<\tau$
\begin{align*}
\e^{\mathbb Q}[f(X_t,I_t)|{\mathcal F}_s]=\frac{1}{Z_s}\e^{\mathbb P}[Z_t f(X_t,I_t)|{\mathcal F}_s]=
\frac{ \e^{\mathbb P}[\phi_{\tau-t}(X_t,I_t)f(X_t,I_t)|X_s, I_s]}{\phi_{\tau-s}(X_s,I_s)}.
\end{align*}
Thus under the measure ${\mathbb Q}$ the process $(X,I)$ has transitional probability density function
$$
\mu^{\mathbb Q}_{s,t}(x_0,z_0;x_1,z_1)=\frac{\phi_{\tau-t}(x_1,z_1)}{\phi_{\tau-s}(x_0,z_0)}\mu^{\mathbb P}_{t-s}(x_0,z_0;x_1,z_1).
$$
The coordinate process $X$ under measure $\mathbb{Q}$ has the law denoted as $\p_X$ in \eqref{eq:Bar-LeD}.
Let $0=t_0<t_1<t_2<\dots<t_n=\tau$. The joint density function  (under ${\mathbb Q}$) of $(X_{t_1}, X_{t_2},...,X_{t_n})$ (given that $X_0=I_0=x_0=z_0=0$) is
\begin{equation}
  \label{eq:density:X}
  \frac{ e^{-\A x_n}}{{\mathfrak K}_{\A,\C}^{(\tau)}} f^{\mathbb Q}_{\mathbf t}({\mathbf x})
\end{equation}
with
\begin{equation*}\label{eq:f}
f^{\mathbb Q}_{\mathbf t}({\mathbf x})=
\int_{0<z_1<z_2<\dots<z_n}  z_n^{-\A/2-\C/2} \prod\limits_{j=1}^n \mu^{\mathbb P}_{t_j-t_{j-1}}(x_{j-1},z_{j-1};x_j,z_j) \d z_1 \d z_2 \dots \d z_n.
\end{equation*}
Here we denoted ${\mathbf x}=(x_1,x_2,\dots,x_n)$ and ${\mathbf t}=(t_1,t_2,\dots,t_n)$.
Changing the variable of integration $z_j-z_{j-1}=w_j$ we have %
\begin{align*} %
f^{\mathbb Q}_{\mathbf t}({\mathbf x})&=
\int_{(0,\infty)^n}  (w_1+w_2+\dots +w_n)^{-\A/2-\C/2} \prod\limits_{j=1}^n \mu^{\mathbb P}_{t_j-t_{j-1}}(x_{j-1},0;x_j,w_j) \d w_1 \d w_2 \dots \d w_n\\ \nonumber
&=
\int_{(0,\infty)^n}  (w_1+w_2+\dots +w_n)^{-\A/2-\C/2}  \\ \nonumber
&\qquad \qquad \prod\limits_{j=1}^n \exp\Big(- \frac{e^{-2x_{j-1}}+e^{-2x_j}}{w_j}\Big)
 \theta\Big(2\frac{e^{-x_{j-1}-x_j}}{w_j},(t_j-t_{j-1})/2\Big) \frac{\d w_j}{w_j}.
\end{align*}
Next we change the variable of integration $\xi_j = 2(e^{-x_{j-1}-x_j})/w_j$ and we obtain
\begin{align*}%
f^{\mathbb Q}_{\mathbf t}({\mathbf x})&=2^{-(\A+\C)/2}
\int_{(0,\infty)^n}  \Big[ \sum\limits_{j=1}^ n \xi_j^{-1} e^{-x_{j-1}-x_j} \Big] ^{-\A/2-\C/2}  \\
& \qquad
\prod\limits_{j=1}^n \exp\Big(- \xi_j  \cosh(x_j-x_{j-1})\Big)
 \theta\Big(\xi_j,(t_j-t_{j-1})/2\Big) \frac{\d \xi_j}{\xi_j}.
\end{align*}

Next we consider a Markov process $Y$ %
with transition probabilities \eqref{Y:trans}
and the initial distribution of $Y_0$ is \eqref{Y0}. Let $0=t_0<t_1<t_2<\dots<t_n=\tau$. The joint distribution of $(Y_0, Y_{t_1},...,Y_{t_n})$ is given by
\begin{align}\label{Y:prejoint}
 \frac{e^{-\C y_0}}{C^{(\tau)}_{\A,\C}} \prod\limits_{j=1}^n p_{t_j-t_{j-1}}(y_{j-1};y_j) e^{-\A y_n} \d y_0 \d y_1 \dots \d y_n.
\end{align} %
Thus the joint density function of $(Y_{t_1}-Y_0,...,Y_{t_n}-Y_0)$ is
\begin{equation}
  \label{eq:density:Y-Y0}
  \frac{e^{-\A x_n}}{C^{(\tau)}_{\A,\C}} \tilde f_{\mathbf t}(\mathbf x)
\end{equation}
with
\begin{equation}\label{eq:tilde_f}
\tilde f_{\mathbf t}(\mathbf x)=\int_{\mathbb R} e^{-(\A+\C) y_0} \prod\limits_{j=1}^n p_{t_j-t_{j-1}}(x_{j-1}+y_0;x_j+y_0)  \d y_0.
\end{equation}
Here we assume that $x_0=0$.
The goal is to show that
\begin{equation}\label{eq:goal}
   \tilde f_{\mathbf t}(\mathbf x) =2^{\A+\C-1} \Gamma((\A+\C)/2) f^{\mathbb Q}_{\mathbf t}(\mathbf x)
\end{equation}
 for all ${\mathbf t}=(t_1,t_2,\dots,t_n)$ and
${\mathbf x}=(x_1,x_2,\dots,x_n)$.
Once \eqref{eq:goal} is established, the normalizing constants will satisfy \eqref{K2C} and hence the densities
\eqref{eq:density:X} and \eqref{eq:density:Y-Y0} will be identical.

It remains to prove \eqref{eq:goal}. To do so, we   express \eqref{eq:tilde_f} in terms of the  Hartman-Watson density function \eqref{Def-HW}.
From \cite[Remark 4.1]{MatsumotoYor2005I}, \cite[formula (50)]{SousaYakubovich2018}, or \cite[formula (7.4)]{Bryc-Kuznetsov-Wang-Wesolowski-2021}
we know that
\begin{align*}
\nonumber p_t(x,y)&=\int_{\mathbb R} \exp\Big(-\frac{1}{2}(e^{x-y-r}+e^{y-x-r}+e^{r-x-y})\Big)  \theta(e^{-r},t/2)d r
\\&=
\int_{0}^{\infty} \exp\Big(-\frac{1}{2}(\xi (e^{x-y}+e^{y-x})+\xi^{-1} e^{-x-y})\Big)  \theta(\xi,t/2)
\frac{d \xi}{\xi} \nonumber
\\&=\int_{0}^{\infty} \exp\Big(-\xi\cosh(x-y) -\tfrac12\xi^{-1} e^{-x-y}\Big)  \theta(\xi,t/2)
\frac{d \xi}{\xi}. %
\end{align*}
Once we plug the latter into \eqref{eq:tilde_f} we get
\begin{align*} %
\tilde f_{\mathbf t}(\mathbf x)&=\int_{\mathbb R}  e^{-(\A+\C) y_0} d y_0
\\ \nonumber & \int_{(0,\infty)^{n}} \prod\limits_{j=1}^n  \exp\Big(-\xi_j\cosh(x_j-x_{j-1})-\tfrac12 \xi_j^{-1} e^{-x_{j-1}-x_j-2y_0}\Big)  \theta(\xi_j,(t_j-t_{j-1})/2)
\frac{\d \xi_j}{\xi_j}.
\end{align*}
The integral  %
with respect to $y_0$ is
\begin{align*}
\int_{\mathbb R} \exp\Big(-(\A+\C)y_0- \tfrac12 e^{-2 y_0} D \Big)\d y_0=D^{-(\A+\C)/2}2^{(\A+\C)/2 -1} \Gamma((\A+\C)/2),
\end{align*}
where $D:=\sum\limits_{j=1}^n \xi_j^{-1} e^{-x_{j-1}-x_j}$.
Thus we obtain
\begin{align*} %
\tilde f_{\mathbf t}(\mathbf x)&=  2^{(\A+\C)/2-1} \Gamma((\A+\C)/2)
\int_{(0,\infty)^{n}}  \Big[\sum\limits_{j=1}^n \xi_j^{-1} e^{-x_{j-1}-x_j} \Big]^{-(\A+\C)/2}
\\ \nonumber
&  \prod\limits_{j=1}^n  \exp\Big(-\xi_j  \cosh(x_j-x_{j-1})\Big)  \theta(\xi_j,(t_j-t_{j-1})/2)
\frac{d \xi_j}{\xi_j}=2^{\A+\C-1} \Gamma((\A+\C)/2) f^{\mathbb Q}_{\mathbf t}(\mathbf x).
\end{align*}
 This establishes \eqref{eq:goal} and completes the proof.

\begin{remark}
An alternative method of proving Theorem \ref{Thm:Y2K} is outlined on page 8 in a recent paper \cite{Corwin2022}.
\end{remark}

\section{Proofs of limit Theorems}
 We begin with some known facts that we will use in several proofs.
We   note that  with $\widetilde Y_t:=Y_{\tau-t}\topp{\A,\C}$, Markov process $(\widetilde Y_t)_{t\in[0,\tau]}$  has the same law as process $(Y_t\topp {\C,\A})_{t\in[0,\tau]}$.
Therefore in some proofs without loss of generality we may assume $\A> 0$.
 We recall \cite[Theorem 3.3]{Bryc-Kuznetsov-Wang-Wesolowski-2021} in our notation/parametrization.  Specialized to the case   $\A>0$ it says that the normalizing constant \eqref{C(ac)} is given by expression
\begin{multline}\label{BKWW:T3.2}
    C_{\A,\C}\topp \tau = \frac{ 2^{\A+\C}}{8\pi}  \int_0^\infty e^{-\tau u^2/4} \frac{|\Gamma(\tfrac{\A+\i u}{2},\tfrac{\C+\i u}{2})|^2}{|\Gamma(\i u)|^2}\d u
  \\+    2^{\A+\C}  \frac{\Gamma(\frac{\C+\A}{2},\frac{\A-\C}{2})}{2\C \Gamma(-\C )}\sum_{\{k\geq 0:\; \C+2k<0 \}} e^{\tau(\C+2k)^2/4}(\C+2k)\frac{(\C,\frac{\A+\C}{2})_k}{k!(1+\frac{\C-\A}{2})_k}.
\end{multline}
From \eqref{C(ac)} is clear that  $C_{\C,\A}\topp \tau =C_{\A,\C}\topp \tau $. It is also clear that
function $(\A,\C,\tau)\mapsto C_{\A,\C}\topp \tau $ is continuous in the domain $\tau>0,\A+\C>0$ in $\r^3$.  In some proofs we will need to determine its behaviour near a point at the boundary.

We will  need small index asymptotics for the modified Bessel K function.
\begin{lemma}
   Assume that $\nu\ne 0$ is an arbitrary complex number. If $x>0$ then
   \begin{equation}
  \label{Alexey-sinh}
  \lim_{T\to +\infty} \frac1{T} K_{\nu/T} (e^{-x T})= \sinh (\nu x)/\nu.
\end{equation}
If $x\leq 0$ then
\begin{equation}
  \label{K20}
  \lim_{T\to +\infty} \frac1{T} K_{\nu/T} (e^{-x T})= 0.
\end{equation}
\end{lemma}
\begin{proof}We know  (\cite[6.8 (26)]{erdelyi1954fg}) that the Mellin transform of Bessel $K$-function is
\begin{equation}
  \label{E6.8(26)} \int_0^{\infty} K_{\nu}(x) x^{s-1} \d x=2^{s-2} \Gamma((s+\nu)/2) \Gamma((s-\nu)/2), \;\;\; \re(s)> \vert\re(\nu)\vert.
\end{equation}
Thus with $x>0$ we can write
\begin{equation*}\label{eq:Besel_K_inverse_Mellin}
K_{\nu}(x)=\frac{1}{8 \pi \i} \int_{c+\i \r}
\Gamma((s+\nu)/2) \Gamma((s-\nu)/2) (x/2)^{-s} \d s,
\end{equation*}
where $c$ is any positive number greater than $\vert \re(\nu)\vert$.
  We shift the contour of integration $c+\i \r \mapsto -1/2+\i \r$. We pick up two residues at poles $s=\pm \nu$, coming from Gamma functions, compute the residues at these poles and obtain
\begin{equation}\label{eq:Besel_K_inverse_Mellin2}
K_{\nu}(x)=\frac{1}{2} \Big( \Gamma(\nu) (x/2)^{-\nu}+
\Gamma(-\nu) (x/2)^{\nu}\Big)+
\frac{1}{8 \pi \i} \int_{-1/2+\i \r}
\Gamma((s+\nu)/2) \Gamma((s-\nu)/2) (x/2)^{-s} \d s.
\end{equation}
In the strip $-3/8<\re(z)<-1/8$ we have a uniform bound $|\Gamma(z)|<C e^{-\frac{\pi}{4} |\im(z)|}$ for some absolute constant $C$ (this is true for any strip of finite width which does not contain the poles of the Gamma function and it follows from Stirling's asymptotic approximation to the Gamma function), thus we can estimate the integral in the right-hand side of \eqref{eq:Besel_K_inverse_Mellin2} as follows
 \begin{equation*}
\bigg \vert \int_{-1/2+\i \r}
\Gamma((s+\nu)/2) \Gamma((s-\nu)/2) (x/2)^{-s} \d s \bigg \vert \le C x^{1/2} \int_{\r} e^{-\frac{\pi}{4} |t|} \d t= x^{1/2}\frac{8 C}{\pi}
 \end{equation*}
 and thus we obtain the following result:
 \begin{equation*}\label{Alexey-K}
 \bigg \vert K_{\nu}(x)- \frac{1}{2} \Big( \Gamma(\nu) (x/2)^{-\nu}+
\Gamma(-\nu) (x/2)^{\nu}\Big) \bigg \vert < \tilde C x^{1/2},
 \end{equation*}
 which is valid for all $x>0$, $\nu \in {\mathbb C}\setminus\{0\}$ with $|\nu|<1/4$ and some absolute constant $\tilde C$.

Since $\Gamma(z)\sim 1/z$ as $|z|\to 0$, this implies that for $x>0$, and  complex $\nu\ne 0$
\begin{equation*}
  \label{Alexey-sinh-a}
  \lim_{T\to +\infty} \frac1{T} K_{\nu/T} (e^{-x T})= \lim_{T\to +\infty} \frac{1}{2 \nu}
\Big( e^{\nu x} \frac{T}{\nu}  \Gamma(\nu/T) 2^{\nu/T}+
e^{-\nu x   }\frac{T}{\nu} \Gamma(-\nu/T) 2^{-\nu/T}\Big)=\sinh (\nu x)/\nu.
\end{equation*}

On the other hand,
for $x \leq 0$,   complex $\nu\ne 0$ and large enough $T$, we have
\begin{equation*}
  \label{K20-a}
   |K_{\nu/T} (e^{-x T})| \leq  K_{|\nu/T|} (e^{-x T})\leq K_{1/2} (e^{-x T}) =
\sqrt{\pi/2}\frac{\exp(-e^{-X T})}{   e^{-x T/2}},
\end{equation*}
which gives \eqref{K20}.
\end{proof}

We will also need an explicit formula for the Doob's $h$-transform \eqref{Ht},
 which relies on analytic extension from $\A>0$.

\begin{lemma}
  \label{Lem:Alexey}
  For   $\A\in\r$, $0\leq t<\tau$ and  $x\in\r$ the expression \eqref{Ht} is finite and is given by
  \begin{multline}
    \label{Ht-explicit}
    H_t(x) =H_t\topp \tau(x):=
    \frac{2^\A}{2\pi} \int_0^\infty e^{ -(\tau-t) u^2/4} K_{\i u}(e^{-x})\frac{ |\Gamma((\A+\i u)/2)|^2 }{|\Gamma(\i u)|^2} \d u
\\  +2^{\A+1}\sum\limits_{k\ge 0 \;  : \; a+2k<0}
  e^{(\tau-t) (\A+2k)^2/4} K_{\A+2k}(e^{-x})\frac{1}{\Gamma(-\A-2k)}.
  \end{multline}
\end{lemma}
\begin{proof}
We fix $\tau>0$ and $x\in \r$. For complex argument $\re(\A)>0$ we define
\begin{equation}\label{eq:G_tau_a0}
G_{\tau}(\A)=\frac{1}{2\pi \i} \int_{\i\r} e^{ \tau w^2/4} K_w(e^{-x})\frac{ \Gamma((\A+w)/2) \Gamma((\A-w)/2)}{\Gamma(w)\Gamma(-w)} \d w.
\end{equation}

 The   integral \eqref{eq:G_tau_a0} converges absolutely: the function $w\mapsto K_{w} (e^{-x})$ is bounded in any strip $w_1<\re(w)<w_2$, the ratio of the gamma functions grows at most as  an exponential function of $|w|$, thus the quadratic exponential factor $e^{\tau w^2/4}$ guarantees convergence.

We use the same method of analytic continuation of
\eqref{BKWW:T3.2} as was used in \cite{Bryc-Kuznetsov-Wang-Wesolowski-2021}. The poles of the integrand occur at points $\pm(\A+2k)$ for $k=0,1,2,\dots$. First we assume that $0<\re(\A)<1$ and shift the contour of integration $\i \r \mapsto 1+ \i \r$, collecting the pole at $w=\A$:
\begin{equation*}\label{eq:G_tau_a1}
G_{\tau}(\A)=\frac{1}{2\pi \i} \int_{1+\i\r} e^{ \tau w^2/4} K_w(e^{-x})\frac{ \Gamma((\A+w)/2) \Gamma((\A-w)/2)}{\Gamma(w)\Gamma(-w)} \d w
+2 e^{\tau \A^2/4} K_{\A}(e^{-x})\frac{1}{\Gamma(-\A)}.
\end{equation*}
The above expression provides analytic continuation in the strip $-1<\re(\A)<1$.
(Recall the convention $1/\Gamma(0)=0$.)
Now we assume that $-1<\re(\A)<0$ and shift the contour of integration back to $\i \r$. We collect the residue at $w=-\A$ and obtain
\begin{equation*}\label{eq:G_tau_a2}
G_{\tau}(\A)=\frac{1}{2\pi \i} \int_{\i\r} e^{ \tau w^2/4} K_w(e^{-x})\frac{ \Gamma((\A+w)/2) \Gamma((\A-w)/2)}{\Gamma(w)\Gamma(-w)} \d w
+4 e^{\tau \A^2/4} K_{\A}(e^{-x})\frac{1}{\Gamma(-\A)}.
\end{equation*}
The above equation gives analytic continuation in the strip $-2<\re(\A)<0$. Continuing this process we obtain an expression for any $\A<0$
\begin{align}\label{eq:G_tau_a3}
G_{\tau}(\A)&=\frac{1}{2\pi \i} \int_{\i\r} e^{ \tau w^2/4} K_w(e^{-x})\frac{ \Gamma((\A+w)/2) \Gamma((\A-w)/2)}{\Gamma(w)\Gamma(-w)} \d w
\\ \nonumber
&+\sum\limits_{k\ge 0 \;  : \; \A+2k<0}
4 e^{\tau (\A+2k)^2/4} K_{\A+2k}(e^{-x})\frac{1}{\Gamma(-\A-2k)}.
\end{align}
Note that the integral \eqref{Ht} is well defined   and  by \eqref{E6.8(26)} formula \eqref{Ht-explicit} holds for  $\re(\A)>0$. We see that
 $H_t(x)=2^{\A-1} G_{\tau-t}(\A)$. Thus formula \eqref{eq:G_tau_a3} shows that for fixed $0\leq t<\tau$ and $x\in\r$, function $\A\mapsto 2^{\A-1} G_{\tau-t}(\A)$ is an analytic extension of \eqref{Ht} to complex plane. Since as a function of $\A$, expression \eqref{Ht} is a Laplace transform when $\re (\A)>0$, its analytic extension \eqref{Ht-explicit} is the Laplace transform of the same non-negative function. (For a version of this fact in the language of analytic characteristic functions, see \cite[Theorem 2]{lukacs1952}.)
\end{proof}

\subsection{Proof of Theorem \ref{T:halfL}}\label{Sect:T:halfL}
We prove convergence of finite dimensional densities by establishing pointwise  convergence of the initial densities, and pointwise
 convergence of the transition densities.

 \subsubsection{Convergence of initial densities}
Since $\A>0$, from \eqref{Ht-explicit} we get
\begin{equation*}\label{H0}
H_0(x)= \frac{2^\A}{2\pi }\int_0^\infty K_{\i u}(e^{-x})e^{-\tau u^2/4} \frac{|\Gamma(\A+\i u)/2|^2}{|\Gamma(\i u)|^2}\d u.
\end{equation*}
So $Y_0\topp{\A,\C}$ has  density
\begin{multline}\label{f}
   f(x)=\frac{2^\A}{2\pi C_{\A,\C}\topp \tau}e^{-\C x}\int_0^\infty K_{\i u}(e^{-x})e^{-\tau u^2/4} \frac{|\Gamma(\A+\i u)/2|^2}{|\Gamma(\i u)|^2}\d u
   \\= \frac{2^\A}{2\pi C_{\A,\C}\topp \tau}e^{-\C x}\int_0^\infty K_{\i v/\sqrt{\tau}}(e^{-x})e^{-v^2/4} \frac{|\Gamma(\A+\i v/\sqrt{\tau})/2|^2}{\sqrt{\tau}|\Gamma(\i v/\sqrt{\tau})|^2}\d v.
\end{multline}
Clearly, $|\Gamma(\A+\i v/\sqrt{\tau})/2|^2\to \Gamma(\A/2)^2$ and  $K_{\i v/\sqrt{\tau}}(e^{-x})\to K_{0}(e^{-x})$  as $\tau\to\infty$.
 It is well known  %
  that
 \begin{equation}
   \label{limit:Gamma}
   \lim_{\tau\to\infty}\frac{\tau}{|\Gamma(\i v/\sqrt{\tau})|^2}=v^2,
 \end{equation}
  and it is clear that we can pass to the limit under the integral sign. So
  \begin{multline*}
     \lim_{\tau\to\infty} f(x)=
   \frac{2^{a}\Gamma(\A/2)^2}{2\pi \lim_{\tau\to\infty}\tau^{3/2} C_{\A,\C}\topp \tau}
  e^{-\C x}  K_0(e^{-x}) \int_0^\infty v^2 e^{-v^2/4}\d v \\ = \frac{2^{a}\Gamma(\A/2)^2}{\sqrt{\pi} \lim_{\tau\to\infty}\tau^{3/2} C_{\A,\C}\topp \tau}
  e^{-\C x}  K_0(e^{-x}).
  \end{multline*}
  To end the proof, we note the following.
\begin{lemma}\label{Lem:C}
$$\lim_{\tau\to\infty}\tau^{3/2} C_{\A,\C}\topp \tau =\begin{cases}
 \frac{1}{4 \sqrt{\pi}}  2^{\A+\C} \Gamma(\A/2)^2\Gamma(\C/2)^2  & \C\geq 0, \\
 \infty   & \C<0 .
\end{cases} $$
\end{lemma}
\begin{proof}
We use \eqref{BKWW:T3.2}. It is clear that the integral is non-negative, and that if $\C<0$ then the discrete sum diverges to $+\infty$ as $\tau\to\infty$, because  the first term of the sum dominates. This gives the second limit in the statement.

If $\C\geq 0$, then there is no discrete sum in \eqref{BKWW:T3.2}, only the integral. We change the variable of integration and write
\begin{equation}
  \label{C*}
  \tau^{3/2} C_{\A,\C}\topp \tau = \frac{2^{\A+\C}}{8\pi}  \int_0^\infty e^{-v^2/4} |\Gamma(\tfrac{\A+\i v/\sqrt{\tau}}{2},\tfrac{\C+\i v/\sqrt{\tau}}{2})|^2\frac{\tau}{|\Gamma(\i v/\sqrt{\tau})|^2}\d v.
\end{equation}
Passing to the limit under the integral sign, from \eqref{limit:Gamma} we get the answer.
\end{proof}

\subsubsection{Convergence of transition densities}
For $\A>0$, from \eqref{Ht-explicit},    we get
\begin{multline*}\label{Ht:expand}
H_t(x)= \frac{2^\A}{2\pi }\int_0^\infty K_{\i u}(e^{-x})e^{-(\tau-t) u^2/4} \frac{|\Gamma(\frac{\A+\i u}{2})|^2}{\Gamma(\i u)|^2}\d u
\\= \frac{2^\A}{2\pi \sqrt{\tau} }\int_0^\infty K_{\i v/\sqrt{\tau}}(e^{-x})e^{-(1-t/\tau) v^2/4} \frac{|\Gamma(\frac{\A+\i v/\sqrt{\tau}}{2})|^2}{\Gamma(\i v/\sqrt{\tau})|^2}\d v,
\end{multline*}
where we used \eqref{E6.8(26)}.
As before,   passing to the limit under the integral sign, for any fixed $0\leq s<t$ and $x,y\in\r$ we get
$$
\lim_{\tau\to\infty}\frac{H_t(y)}{H_s(x)}= \frac{K_0(e^{-y})}{K_0(e^{-x})}.
$$
Therefore the transition probabilities converge (pointwise) to \eqref{Z:trans}. \qed

\subsection{Proof of Theorem \ref{Thm:KPZ:fixed}} \label{Sect:KPZ:fixed} Passing to the process $(Y_{\tau-t})_{t\in [0,\tau]}$ if necessary, without loss of generality  we may assume that $\A>0$.
Recall \eqref{Y:prejoint}.
For $0=t_0<t_1<\dots<t_n=1$, the joint density of vector $\frac{1}{\sqrt{\tau}}\left(Y_{\tau t_j}\topp{\A/\sqrt{\tau},\C/\sqrt{\tau}} \right)_{j=0,\dots,n}$  is
\begin{equation}
  \label{Y:joint}\frac{\sqrt{\tau}}{C_{\A/\sqrt{\tau},\C/\sqrt{\tau}}\topp \tau} e^{-\C x_0}H_0(x_0/\sqrt{\tau})e^{-\A x_n} \prod_{k=1}^n \sqrt{\tau}p_{\tau(t_k-t_{k-1})}(\sqrt{\tau} x_{k-1},\sqrt{\tau}x_k).
\end{equation}
Since the joint density of $\left(\widetilde \eta_{t_j}\topp{\A,\C} \right)_{j=0,\dots,n}$  is of similar product form
\begin{equation}
  \label{eta:joint}
  \frac{1}{\mathfrak C_{\A,\C}}e^{-\C x_0} h_0(x_0) 1_{x_0>0} e^{-\A x_n}\prod_{k=1}^n g_{t_{k-1}-t_k}(x_{k-1},x_k)1_{x_k>0},
\end{equation}
we only need to prove convergence of the corresponding factors.  With $\A>0$, the part of the first factor  that depends on $\tau$ is
\begin{multline}\label{f:tau}
f_\tau(x):= \frac{\sqrt{\tau}}{C_{\A/\sqrt{\tau},\C/\sqrt{\tau}}\topp \tau}  H_0(x/\sqrt{\tau}) =  \frac{\sqrt{\tau}2^{\A/\sqrt{\tau}}}{2\pi C_{\A/\sqrt{\tau},\C/\sqrt{\tau}}\topp \tau}
\int_0^\infty K_{\i u}(e^{-x\sqrt{\tau}})e^{-\tau u^2/4} \frac{|\Gamma(\frac{\A/\sqrt{\tau}+\i u}{2})|}{\Gamma(\i u)|^2}\d u
\\=2^{\A/\sqrt{\tau}} \frac{\sqrt{\tau}}{2\pi C_{\A/\sqrt{\tau},\C/\sqrt{\tau}}\topp \tau}
\int_0^\infty \frac{K_{\i v/\sqrt{\tau}}(e^{-x\sqrt{\tau}})}{\sqrt{\tau}}e^{-v^2/4} \frac{|\Gamma(\frac{\A +\i v}{2\sqrt{\tau}})|}{\Gamma(\i v/\sqrt{\tau})|^2}\d v,
\end{multline}
 compare \eqref{f}.
 In the next lemmas, we verify that the integral and the multiplicative constant in the above expression converge.

 Recall \eqref{C:norm:eta}. We have
 \begin{lemma}\label{Llem:C2}  %
 $$
 \lim_{\tau\to\infty}  \frac{  C_{\A/\sqrt{\tau},\C/\sqrt{\tau}}\topp \tau}{\sqrt{\tau}} =  \mathfrak C_{\A,\C}.$$
\end{lemma}
\begin{proof} We invoke  \eqref{BKWW:T3.2} with $\A/\sqrt{\tau}$ and  $\C/\sqrt{\tau}$.
  For large $\tau$, at most one atom may be present.
With $\A>0$,
 the contribution of this atom to $  C\topp \tau_{\A/\sqrt{\tau},\C/\sqrt{\tau}}$ is
  $$
  2^{(\A+\C)/\sqrt{\tau}}
  \frac{\Gamma(\frac{\C+\A}{2\sqrt{\tau}},\frac{\A-\C}{2\sqrt{\tau}})}{2 \Gamma(-\C/\sqrt{\tau} )}  e^{\C^2/4} 1_{\C<0}.
 $$
 Combining this with the integral part, we get
   \begin{multline*}
   \frac{2\pi  C\topp \tau_{\A/\sqrt{\tau},\C/\sqrt{\tau}}}{\sqrt{\tau}}=\frac{1}{4\sqrt{\tau}}\int_0^\infty e^{-\tau u^2/4} \frac{|\Gamma(\tfrac{\A/\sqrt{\tau}+\i u}{2},\tfrac{\C/\sqrt{\tau}+\i u}{2})|^2}{|\Gamma(\i u)|^2}\d u +
   \pi   2^{(\A+\C)/\sqrt{\tau}}
  \frac{\Gamma(\frac{\C+\A}{2\sqrt{\tau}},\frac{\A-\C}{2\sqrt{\tau}})}{\sqrt{\tau} \Gamma(-\C/\sqrt{\tau} )}  e^{\C^2/4}1_{\C<0}
   \\= \frac{1}{4}\int_0^\infty e^{-v^2/4} \frac{|\Gamma(\tfrac{\A+\i v}{2\sqrt{\tau}},\tfrac{\C+\i v}{2\sqrt{\tau}})|^2}{\tau|\Gamma(\i v/\sqrt{\tau})|^2}\d v  +
   \pi   2^{(\A+\C)/\sqrt{\tau}}
  \frac{\Gamma(\frac{\C+\A}{2\sqrt{\tau}},\frac{\A-\C}{2\sqrt{\tau}})}{\sqrt{\tau} \Gamma(-\C/\sqrt{\tau} )}  e^{\C^2/4} 1_{\C<0} \\
   \to \int_0^\infty e^{-v^2/4} \frac{4 v^2}{(\A^2+ v^2) (\C^2+v^2)}\d v - \frac{4\pi \C e^{\C^2/4}}{\A^2-\C^2}1_{\C<0},
   \end{multline*}
   as we can pass to the limit  under the integral sign.

     Here we use $\Gamma(z)\sim 1/z$ as $z\to 0$ which gives
\begin{equation*}
  \label{Mthmtica:III}
   \lim_{\tau\to\infty} \frac{|\Gamma(\tfrac{\A+\i v}{2\sqrt{\tau}},\tfrac{\C+\i v}{2\sqrt{\tau}})|^2}{\tau|\Gamma(\i v/\sqrt{\tau})|^2}
  =\frac{16 v^2}{(\A^2+v^2)(\C^2+v^2)} \mbox{ and } \lim_{\tau\to\infty}  \frac{\Gamma(\frac{\C+\A}{2\sqrt{\tau}},\frac{\A-\C}{2\sqrt{\tau}})}{2 \sqrt{\tau} \Gamma(-\C/\sqrt{\tau} )}=\frac{-2\C}{\A^2-\C^2}.
\end{equation*}
To complete the proof we need the following, where in view of symmetry we assume $\A>0$.
\end{proof}
\begin{lemma} For $\A+\C>0$ with $\A>0$, we have
\begin{equation}\label{C-no-CC}
    \frac{1}{2\pi}\int_0^\infty e^{-v^2/4} \frac{4 v^2}{(\A^2+ v^2) (\C^2+v^2)}\d v - \frac{2 \C e^{\C^2/4}}{\A^2-\C^2}1_{\C<0}= \mathfrak C_{\A,\C}.
\end{equation}
\end{lemma}
\begin{proof} 
   For non-zero $\A,\C$ the integral in \eqref{C-no-CC} is
   \begin{multline*}
      \frac{1}{2\pi}\int_0^\infty e^{-v^2/4} \frac{4 v^2}{(\A^2+ v^2) (\C^2+v^2)}\d v =
       \frac{2}{\pi}\int_0^\infty e^{-v^2/4} \int_0^\infty e^{-|\A| x}\sin (v x) \d x \int_0^\infty e^{-|\C| y} \d y \sin (v y)\d v
      \\=   \frac{2}{\pi} \int_0^\infty \int_0^\infty e^{-|\A| x-|\C| y} \int_0^\infty e^{-v^2/4} \sin (v x)   \sin (v y)\d v  \d x \d y =     \int_0^\infty e^{-|\A| x-|\C| y}  g_1(x,y)  \d x \d y\\=  \mathfrak C_{|\A|,|\C|},
   \end{multline*}
  where we used \eqref{eta:heat} and Fubini's theorem. By taking a limit (or modifying the above  calculation) the formula extends to $\C=0$.
  This shows that  formula \eqref{C-no-CC} holds for $\A>0,\C\geq 0$.

  To extend the formula to $\C\in(-\A,0)$  we use the  identity $\erfc(x)+\erfc(-x)=2 $  and explicit form
  of   \eqref{C:norm:eta}, i.e.,
\begin{equation*}
  \label{CC(a,c)}
\mathfrak C_{\A,\C}=
\begin{cases}\tfrac{\A e^{\A^2/4}\erfc(\A/2)-\C e^{\C^2/4}\erfc(\C/2)}{\A^2-\C^2}& \A\ne \C , \A+\C>0 \\ \\
\frac{2+\A^2/2}{4\A}e^{\A^2/4}\erfc(\A/2)-\tfrac{1}{2\sqrt{\pi}} &\A=\C>0.
\end{cases}
\end{equation*}
  A calculation    verifies that  %
$$\mathfrak C_{\A,\C}=\mathfrak C_{\A,-\C}- \frac{2 \C e^{\C^2/4}}{\A^2-\C^2}.$$
\end{proof}

Next we tackle convergence of the integral in \eqref{f:tau}.
The factor $e^{-v^2/4}$ allows us to pass to the limit under the integral sign.
We use
\begin{equation*}
  \label{Gamma1} \lim_{\tau\to\infty}\frac{|\Gamma((\A +\i v)/(2\sqrt{\tau}))|^2}{|\Gamma(\i v/\sqrt{\tau})|^2}=\frac{4 v^2}{\A^2+v^2}
\end{equation*}
and the following.
For $v>0$,
\begin{equation}
  \label{limK}
  \lim_{\tau\to\infty}\frac{K_{\i v/\sqrt{\tau}}(e^{-x\sqrt{\tau}})}{\sqrt{\tau}}= \begin{cases}
   \frac{\sin( v x)}{v} & x>0, \\
   0& x\leq 0.
 \end{cases}
\end{equation}
This follows from \eqref{K20} when  $x\leq 0$, and from \eqref{Alexey-sinh} with $\nu=\i v$ when $x>0$.
 For $x>0$, passing to the limit under the integral and invoking \eqref{limK}, we get
\begin{multline*}
 e^{-\C x} \lim_{\tau\to \infty} f_\tau(x)= \frac{e^{-\C x}}{2\pi\mathfrak C_{\A,\C}} {\mathbf 1}_{\{x>0\}}\int_0^\infty
  e^{-v^2/4} \frac{4v}{v^2+\A^2} \sin ( v x) \d v  \\=
   \frac{e^{-\C x}}{\mathfrak C_{\A,\C}} {\mathbf 1}_{\{x>0\}}\int_0^\infty \frac{2}{\pi} \int_0^\infty
  e^{-v^2/4} e^{-\A y} \sin (v y)  \sin ( v x) \d v \d y
  = \frac{e^{-\C x}}{\mathfrak C_{\A,\C}}{\mathbf 1}_{\{x>0\}} \int_0^\infty  e^{-\A y} g_1(x, y)\d y
  \\= \frac{e^{-\C x}h_0(x)}{\mathfrak C_{\A,\C}},
\end{multline*}
which matches \eqref{eta0} and recovers the first factor in \eqref{eta:joint}.

 Next, we consider a single factor from the product expression in \eqref{Y:joint}. For $t>0$ we have
\begin{multline*}
 \sqrt{\tau} p_{\tau t}(x\sqrt{\tau},y\sqrt{\tau}) = \sqrt{\tau}\frac{2}{\pi}\int_0^\infty e^{ -t\tau u^2}K_{\i u}(e^{x\sqrt{\tau}})K_{\i u}(e^{y\sqrt{\tau}})\frac{\d u}{|\Gamma(\i u)|^2}
 \\=  \frac{2}{\pi}\int_0^\infty e^{ -t v^2}\frac{1}{\sqrt{\tau}}K_{\i v/\sqrt{\tau}}(e^{x\sqrt{\tau}})\frac{1}{\sqrt{\tau}}K_{\i v/\sqrt{\tau}}(e^{y\sqrt{\tau}})\frac{\tau}{|\Gamma(\i v/\sqrt{\tau})|^2}\d v.
\end{multline*}
 So  by the dominated convergence theorem, \eqref{limit:Gamma} and \eqref{limK} give
\begin{equation}
  \label{lim:p2g}
  \lim_{\tau\to\infty} \sqrt{\tau} p_{ \tau t}(x\sqrt{\tau},y\sqrt{\tau})
= \begin{cases}
\frac{2}{\pi}\int_0^\infty e^{ -t v^2/4}\sin (x v) \sin(y v) \d v=g_t(x,y)& \mbox{ if } x>0,y>0,\\
\\
0 &\mbox{otherwise}.
\end{cases}
\end{equation}
Taking it all together, we see that density \eqref{Y:joint} converges to density \eqref{eta:joint} for all $(x_0,\dots,x_n)\in\r^{n+1}$.

\begin{remark}\label{Rem:Yizao}
As expected,   process $\widetilde \eta\topp{\A,\C}$ is a ``fixed point" of the procedure %
in the following sense.
Instead of taking  $t\in [0,1]$, consider
   a Markov process $(\widetilde \eta_t\topp{\A,\C})_{t\in[0,\tau]}$   with initial distribution \eqref{eta0} with the normalizing constant $\int_{\r_+^2} e^{-\C x - \A y}g_\tau(x,y)\d x\d y$, and transition probabilities \eqref{eta:trans} but with
  Doob's $h$-transform given by
$$h_t(x)= \int_{\r_+}g_{\tau-t}(x,y)e^{-\A y} \d y, \; 0\leq t<\tau $$
(with $h_\tau(x):=e^{-\A x}$).
Then
 the law of
  $\frac{1}{\sqrt{\tau}}\{\widetilde \eta\topp{\A/\sqrt{\tau},\C/\sqrt{\tau}}_{\tau t}\}_{t\in [0,1]}$ does not depend on $\tau>0$.
  This is a consequence of scaling  $\sqrt{\tau} g_{\tau t}(x\sqrt{\tau},y\sqrt{\tau})=g_t(x,y)$ for the kernel \eqref{eta:heat}.
  \end{remark}

\subsection{Proof of Theorem \ref{T:halfL:hd}}\label{Sec:halfL:hd}
To make the dependence on $\tau$ explicit, we write $H_t\topp\tau(x)$ for expression \eqref{Ht-explicit}.
 We fix $0=t_0<t_1<\dots <t_n$.
The joint distribution of the vector $(Y_0,Y_{t_1},\dots,Y_{t_n})$ has density
\begin{equation}
  \label{f-tau}
  f_\tau(\vv x)=\frac{H_{t_n}\topp \tau(x_{n})}{C_{\A,\C}\topp \tau} e^{-\C x_0}\prod_{j=1}^n p_{t_j-t_{j-1}}(x_{j-1},x_j).
\end{equation}

We note that
as $\tau \to \infty$ the integral term in \eqref{Ht-explicit}  converges to $0$ and if $\A<0$ then the dominant term in the finite sum is the one with $k=0$.
Thus we have the following result: if $\A<0$ then
\begin{equation*}
H_t\topp\tau(x)=2^{\A+1} e^{(\tau-t) \A^2/4} \frac{ K_{\A}(e^{-x})}{\Gamma(-\A)} (1+o(1)), \;\;\; \tau \to +\infty.
\end{equation*}
In the same way, from \eqref{BKWW:T3.2}  we see that when $\A<0$ then
\begin{equation*}
C_{\A,\C}^{(\tau)}=2^{\A+\C-1} \frac{\Gamma((\C+\A)/2)\Gamma((\C-\A)/2)}{\Gamma(-\A)} e^{\tau \A^2/4}(1+o(1)), \;\;\; \tau \to +\infty.
\end{equation*}
Thus we obtain
 \begin{equation}
   \label{H/C}
   \lim_{\tau\to\infty}\frac{ H_{t}\topp \tau(x)}{C_{\A,\C}\topp \tau}=\frac{4 e^{-t\A^2/4}K_a(e^{-x})}{2^c\Gamma((\A+\C)/2)\Gamma((\C-\A)/2)}.
 \end{equation}
If $\A=0$, changing the variable in the integral \eqref{Ht-explicit}, compare \eqref{C*}, we get
$$
H_t\topp\tau(x)=\frac{2}{\sqrt{\pi}}K_0(e^{-x})\left(\frac{1}{\sqrt{\tau}}+o(1/\sqrt{\tau})\right),\;\tau\to\infty $$
and
$$C_{0,\C}^{(\tau)}=\frac{2^c}{2\sqrt{\pi}}\Gamma(\C/2)^2\left(\frac{1}{\sqrt{\tau}}+o(1/\sqrt{\tau})\right),\;\tau\to\infty .
$$
(Here we use $4\int_0^\infty \exp(-v^2/4)\d v=4\sqrt{\pi}$.) Thus \eqref{H/C} holds also for $\A=0$.

This shows that the density \eqref{f-tau} converges as $\tau\to\infty$ to the joint density
$$
\frac{4 }{2^c\Gamma((\A+\C)/2)\Gamma((\C-\A)/2)}e^{-\C x_0}\left(\prod_{j=1}^n p_{t_j-t_{j-1}}(x_{j-1},x_j)\right)
e^{-t_n\A^2/4}K_a(e^{-x_n})$$
of  vector
$(Z_0\topp{\A,\C},Z_{t_1}\topp{\A,\C},\dots,Z_{t_n}\topp{\A,\C})$.

\subsection{Proof of Theorem \ref{T:Alexey2}}\label{Sec:Alexey2}
Let $z>0$. Following \cite{donati2001some} and \cite{hariya2004limiting}, consider the transformation
$$
{\mathbb T}_z(X)(t)=X_t-\ln\Big(1+z \int_0^{t} e^{2 X_s} d s\Big)
$$
that maps continuous function $\{X_s\}_{0\le s \le t}$ into a continuous function.

For $\nu \in \r$ and $\alpha>0$ we define the Generalized Inverse Gaussian distribution
$$
\p({\textrm{GIG}}(\nu,\alpha) \in \d x)=\frac{\alpha^{-\nu} x^{\nu-1}}{2 K_{\nu}(\alpha)}
e^{-\frac{1}{2}(x+\alpha^2/x)}{\mathbf 1}_{\{x>0\}} \d x,
$$
which in statistical literature \cite{jorgensen2012statistical}   comes with an additional scale parameter.
From Proposition 4.1 in \cite{hariya2004limiting},
 we get the following result: for any $\alpha>0$, $\nu \in \r$, $t\ge 0$ and every non-negative ${\mathcal F}_t$-measurable functional $F$ one has
\begin{equation}\label{HY-P4.1}
{\mathbb E}[ F({\mathbb T}_{{\textrm{GIG}}(\nu,\alpha)}(W^{(\nu)})(t), \; s\le t)]=
e^{-\nu^2 t/2}{\mathbb E}\left[F(W_s, \; s\le t) e^{-\frac{\alpha^2}{2} \int_0^{t} e^{2 W_s} d s}
\frac{K_{\nu}(\alpha e^{W_t})}{K_{\nu}(\alpha)} \right].
\end{equation}
Here $W^{(\nu)}_t=W_t+\nu t$ is a standard Brownian motion with drift $\nu t$. We set $\nu=-\A$, $\alpha=e^{-x}$, scale time parameter $t\mapsto t/2$, use the fact that $(W_t)_{t\ge 0}$ and $(-W_t)_{t\ge 0}$ have the same distribution (when the Brownian motion is started from zero) and that the right hand side of \eqref{HY-P4.1} is invariant with respect to changing $\nu$ to $-\nu$.
Applying the above identity to functional $F(-W_s, s\le t)$, we get
\begin{multline}\label{eq:Haria_Yor3}
{\mathbb E}\left[ F(-{\mathbb T}_{\frac{1}{2}{\textrm{GIG}}(-\A,e^{-x})}(-X^{(\A)})(t), \; s\le t)\right]\\=
{\mathbb E}\left[F(X_s-x, \; s\le t) e^{-\frac{1}{4} \int_0^{t} e^{-2 X_s} d s}
\frac{e^{-\A^2 t/4}K_{\A}(e^{-X_t})}{K_{\A}(e^{-x})} \middle| X_0=x\right],
\end{multline}
where %
$(X^{(\nu)}_t):=(W^{(\nu)}_{t/2})\stackrel{d}{=}(\frac{1}{\sqrt{2}}W_t+\nu t/2)\stackrel{d}{=}(B_t+\nu t/2)$ and
$(X_t):=(X^{(0)}_t)\stackrel{d}{=}(\frac{1}{\sqrt{2}} W_t)\stackrel{d}{=}(B_t)$.
Let $\tilde X$ be the Markov process having transition probability density \eqref{pt}.
As was discussed in \cite[Section 3]{Bryc-Kuznetsov-Wang-Wesolowski-2021}, $\tilde X$ can be identified with the Brownian motion $B$ (of variance   $1/2$) killed at a rate $\frac{1}{4} e^{-2B_t}$.
 More precisely,  the semigroup of the process $\tilde X$ is given by
\begin{equation*}\label{P_t-semi}
\tilde {\mathcal {P}}_t f(x)=\e_x[f(\tilde X_t)]=\e\left[ e^{-\tfrac14\int_0^t e^{-2 B_s} \d s} f( B_t) \middle| B_0=x \right].
\end{equation*}

 Thus we can rewrite \eqref{eq:Haria_Yor3} in the form

\begin{equation}\label{eq:Haria_Yor5}
{\mathbb E}\left[ F(-{\mathbb T}_{\frac{1}{2}{\textrm{GIG}}(-\A,e^{-x})}(-X^{(\A)})(t), \; s\le t)\right]=
e^{-\A^2 t/4}{\mathbb E}\left[F(\tilde X_s-x, \; s\le t)
\frac{K_{\A}(e^{-\tilde X_t})}{K_{\A}(e^{-x})} \middle \vert \tilde X_0=x\right].
\end{equation}

Now we start the process $\tilde X$ from initial distribution
$$
\p(\tilde X_0=\d x)=\frac{4 e^{-\C x} K_{\A}(e^{-x})}{2^{\C} \Gamma((\C-\A)/2) \Gamma((\C+\A)/2)}\d x
$$
and we note that this gives us the process $Z\topp{\A,\C}$ in Theorem \ref{T:halfL:hd}. So  \eqref{eq:Haria_Yor5} becomes
\begin{equation}\label{eq:Haria_Yor4}
\int {\mathbb E}[ F(-{\mathbb T}_{\frac{1}{2}{\textrm{GIG}}(-\A,e^{-x})}(-X^{(\A)})(t), \; s\le t)]
\p(\tilde X_0=\d x)= {\mathbb E}\Big[F( Z_s\topp{\A,\C}-Z_0\topp{\A,\C}, \; s\le t)\Big].
\end{equation}
On the left-hand side of \eqref{eq:Haria_Yor4} we get:
\begin{align*}
\int_{\r} \frac{4 e^{-\C x} K_{\A}(e^{-x})}{2^{\C} \Gamma((\C-\A)/2) \Gamma((\C+\A)/2)}
{\mathbb E}[ F(-{\mathbb T}_{\frac{1}{2}{\textrm{GIG}}(-\A,e^{-x})}(-X^{(\A)})(t), \; s\le t)]\d x.
\end{align*}
We need to compute the distribution of the mixture
$$
\int_{\r} \frac{4 e^{-\C x} K_{\A}(e^{-x})}{2^{\C} \Gamma((\C-\A)/2) \Gamma((\C+\A)/2)}
\p(\tfrac{1}{2}{\textrm{GIG}}(-\A,e^{-x})\in \d y) \d x.
$$
The density of this mixture of distributions can be written in the form
\begin{align*}
\int_{\r} \frac{4 e^{-\C x} K_{\A}(e^{-x})}{2^{\C} \Gamma((\C-\A)/2) \Gamma((\C+\A)/2)}
\frac{e^{-\A x} (2y)^{-\A-1}}{K_{\A}(e^{-x})} e^{-y-e^{-2 x}/(4y)} \d x=
\frac{y^{(\C-\A)/2-1}}{\Gamma((\C-\A)/2)} e^{-y} 1_{y>0}
\end{align*}
and this is the density of $\gamma_{(\C-\A)/2}$ random variable. Thus we obtain
\begin{align*}
 {\mathbb E}\Big[F( Z_s\topp{\A,\C}-Z_0\topp{\A,\C}, \; s\le t)\Big]=
{\mathbb E}[ F(-{\mathbb T}_{\gamma_{(\C-\A)/2}}(-X^{(\A)})(t), \; s\le t)],
\end{align*}
which implies that the process $(Z_t\topp{\A,\C}-Z_0\topp{\A,\C})$ has the same distribution as
$$
-{\mathbb T}_{\gamma_{(\C-\A)/2}}(-X^{(\A)})(t)=X^{(\A)}_t+\ln\Big(1+\gamma_{(\C-\A)/2}
\int_0^{t} e^{-2 X^{(\A)}_s} d s \Big),
$$
where %
$(X^{(\A)}_t)\stackrel{d}{=}(\frac{1}{\sqrt{2}}W_t+\A t/2)$
 and $W$ is the standard Brownian motion. Using the fact that the distribution of $(Z_t\topp{\A,\C})$ is invariant under change $\A \mapsto -\A$, we obtain the expression in \eqref{KPZ:half:hd} and thus conclude the proof of Theorem
 \ref{T:Alexey2}. %
  \subsection{Proof of Theorem \ref{T:Z:fixed}}\label{Sec:T:Z:fixed}
  Recall that the $1/\sqrt{2}$ multiple of the  3-dimensional Bessel process  $BES^{3}$    \cite[Ch VI \$3]{revuz2013continuous}
  has transition probabilities
\begin{equation}\label{KPZ:fixed:tr}
  \p(\rho_t\topp \C=\d y|\rho_s\topp \C=x)= \frac{y}{x} g_{t-s}(x,y)\d y,\; 0\leq s<t,\; x,y>0,
\end{equation}
with kernel \eqref{eta:heat}.

    \begin{proof}[Proof of Theorem \ref{T:Z:fixed}]

It is known \cite[(10.25.3) and (10.30.3)]{NIST2010}  that %
\begin{equation*}
  K_0 (z)\sim \sqrt{\frac{\pi}{2 z}}e^{-z} \mbox{ as $z\to \infty$, and  }
  K_0(z)\sim -\log z \mbox{ as $z\searrow 0$}.
\end{equation*}
We get $$\lim_{T\to \infty}   K_0(e^{-  x \sqrt{T}})/\sqrt{T}= x {\mathbf 1}_{\{x>0\}}.$$ Using this in \eqref{Z0} we see that the densities of $Z_0\topp { \C/\sqrt{T}}$ converge pointwise. We get
    $\tfrac{1}{\sqrt{T}}Z_0\topp {\C/\sqrt{T}}\Rightarrow  \C^2 x e^{-\C x}{\mathbf 1}_{\{x>0\}}\d x$,
which is the density of $\frac{1}{\C}\gamma_2$.
From \eqref{Z:trans} and \eqref{lim:p2g}, we see that the transition densities
$$
\frac{K_0(e^{-y\sqrt{T}})}{K_0(e^{-x\sqrt{T}})}\;\sqrt{T} p_{T(t-s)}(x\sqrt{T},y\sqrt{T})
$$
for the process $\frac{1}{\sqrt{T}}\left(Z_{tT}\topp {\C/\sqrt{T}}\right)_{t\geq 0} $
converge to \eqref{KPZ:fixed:tr} when $x,y>0$.

A minor technical issue is an undefined expression 0/0 in case $x<0$, which can be handled by considering joint multivariate density as in proof of Theorem \ref{Thm:KPZ:fixed}. We omit the details.
  \end{proof}

\subsection{Proof of Theorem \ref{T:half:hd:fixed}}\label{Sec:half:hd:fixed}
Recall that the initial law \eqref{Z0:hd} is the law of $-\log\left(2 \sqrt{\gamma_{\tfrac{\A+\C}{2}} \tilde \gamma_{\tfrac{\C-\A}{2}}}\right)$.
Since $-\eps \log \gamma_\eps\Rightarrow \gamma_1$ as $\eps\to0$, the initial law $\mu_{\A,\C}(dx)$ is the law of
$\frac{1}{\A+\C}\gamma_1+\frac1{\C-\A}\tilde \gamma_1$.
In particular, when $\A=0$ here we get \eqref{KPZ:fixed:Z0}.

For transition probabilities, we would only need to find the asymptotics of
$e^{-\A^2 t}K_{\A/\sqrt{T}} \left(e^{-x \sqrt{T}}\right)$. For $\A\ne 0$, the limit is
\begin{equation*}
  \lim_{T\to +\infty} \frac{1}{\sqrt{T}} K_{\A/\sqrt{T}}\left(e^{-\sqrt{T} x}\right) =
  \begin{cases}
     \frac{1}{\A} \sinh(\A x) & x>0 ,\\
     0 & x\leq 0.
  \end{cases}
\end{equation*}
For $x\leq 0$, the limit follows from \eqref{K20}. For $x>0$ we use
\eqref{Alexey-sinh} with $\nu=\A$.
To avoid  undefined expression 0/0 when $x\leq 0$, we need to consider joint multivariate density as in the proof of Theorem \ref{Thm:KPZ:fixed}. We omit the details.

 \subsection{Proofs of the observations about convergence of univariate laws}
In this section we collect proofs of observations on the limits of the initial laws which served as justification that some of the limits in \cite{barraquand2022steady} do not extend to the Markov process level.
A convenient tool for this task are identities for the
Laplace transform, which we state in slightly more general form than what we need.

Recall \eqref{C(ac)}.
Since the joint law of $(Y_0,Y_\tau)$ is $$\frac{1}{{C_{\A,\C}\topp \tau}}e^{-\C x-\A y}p_\tau(x,y)\d x \d y,$$ we have
\begin{equation*}
  \label{Jacek0} \e \left[\exp\left(  s_0 Y_\tau\topp{\A,\C} +s_1Y_\tau\topp{\A,\C}\right)\right]=\frac{C_{\A-s_1,\C-s_0}\topp \tau}{C_{\A,\C}\topp \tau}.
\end{equation*}
  This gives  %
\begin{equation}\label{Jacek1}
  \e \left[\exp\left( - s Y_0\topp{\A,\C}\right)\right]=\frac{C_{\A,s+\C}\topp \tau}{C_{\A,\C}\topp \tau}.
\end{equation}

\subsubsection{Proof of  (\ref{a+b=0:1}) (convergence to the exponential law)}%
\label{Sec:P-a+b=0}  %

 By symmetry,  if $\A\ne 0$ without loss of generality we may assume that $\A<0$ so $\min\{\A_\eps,\C_\eps\}<0$  for all $\eps>0$ small enough.

  We now %
  use  \eqref{Jacek1} to
   verify that  $\eps Y_0\topp{\A_\eps,\C_\eps}\Rightarrow \gamma_1$.
  Of course $\C_\eps\to-\A$ so without loss of generality, we assume $\A>0$. (The proof requires minor modifications if $\A<0$.)
From \eqref{BKWW:T3.2}, as $\eps\to0$ we get
\begin{multline*}
    C_{\A_\eps,\C_\eps+\eps s}\topp \tau \sim  \frac{ 1}{8\pi}  \int_0^\infty e^{-\tau u^2/4} \frac{|\Gamma(\tfrac{\A+\i u}{2},\tfrac{-\A +\i u}{2})|^2}{|\Gamma(\i u)|^2}\d u
  \\+       \frac{\Gamma(\frac{(1+s)\eps}{2},\A)}{2\C \Gamma(\A )}\sum_{\{k\geq 0:\; -\A+2k<0 \}} e^{\tau(\C+2k)^2/4}(\C+2k)\frac{(\C,\frac{\eps}{2})_k}{k!(1-\A)_k}.
\end{multline*}
Since the integral converges and $\Gamma(x)\sim \tfrac{1}{x}-\gamma$ as $x\to 0$, the leading term in the asymptotics comes from the coefficient  in front of the sum and the first term in the sum. We see that
\begin{align*}
\e \left[\exp\left( - s \eps Y_0\topp{\A_\eps,\C_\eps}\right)\right]=\frac{C_{\A_\eps,\C_\eps+\eps s}\topp \tau}{C_{\A_\eps,\C_\eps}\topp \tau}
&\sim   \frac{\Gamma(\frac{\C_\eps+\A_\eps+\eps s}{2},\frac{\A_\eps-\C_\eps-\eps s}{2})}{ \Gamma(-\C_\eps -\eps s)}  e^{\tau(\C_\eps+\eps s)^2/4}
\times \Bigg[\frac{\Gamma(\frac{\C_\eps+\A_\eps}{2},\frac{\A_\eps-\C_\eps}{2})}{ \Gamma(-\C_\eps )}e^{\tau\C_\eps^2/4}\Bigg]^{-1}\\
&\sim \frac{\Gamma((1+s)\eps/2)}{\Gamma(\eps/2)}\to \frac{1}{1+s},
\end{align*}
which is the Laplace transform of the exponential $\gamma_1$ law.
\qed

\subsubsection{Proof of asymptotic normality (\ref{Half:line:N})}\label{Sec:HalfZ:II}
We note that as indicated in \eqref{Y0}, the law of $Y_0\topp{\A,\C}$ depends on $\tau$ though the normalizing constant and though $H_0$.
\begin{proof}[Proof of \eqref{Half:line:N}]

We   use \eqref{Jacek1} to determine the Laplace transform.
\begin{equation*}
L_\tau(s):=  \e \left[ e^{-s (Y_0\topp{\A,\C}+\tau \C/2)/\sqrt{\tau}}\right]=e^{-s \C \sqrt{\tau}/2}\frac{C_{\A,\C+s/\sqrt{\tau}}\topp \tau}{C_{\A,\C}\topp \tau}.
\end{equation*}
It is enough to determine how the numerator depends on $s$ for $s<0$, discarding all multiplicative constants that will cancel out with the denominator.
From \eqref{BKWW:T3.2} we know that $ C_{\A,\C+\eps s/\sqrt{\tau}}\topp \tau=I_\tau(s)+D_\tau(s)$ is the  sum of the integral
\begin{align*}
I_\tau(s)&=  \frac{ 1}{8\pi}  \int_0^\infty e^{-\tau u^2/4} \frac{|\Gamma(\tfrac{\A+\i u}{2},\tfrac{\C +s/\sqrt{\tau}+\i u}{2})|^2}{|\Gamma(\i u)|^2}\d u
  \\ &\leq  \frac{ 1}{8\pi}  \int_0^\infty e^{- u^2} \frac{|\Gamma(\tfrac{\A+\i u}{2},\tfrac{\C +s/\sqrt{\tau}+\i u}{2})|^2}{|\Gamma(\i u)|^2}\d u \to \frac{ 1}{8\pi}  \int_0^\infty e^{- u^2} \frac{|\Gamma(\tfrac{\A+\i u}{2},\tfrac{\C  +\i u}{2})|^2}{|\Gamma(\i u)|^2}\d u,
\end{align*}
which is bounded in $\tau$, and the finite sum
\begin{align*}
   D_\tau(s)&= 2^{\A+\C+s/\sqrt{\tau}}  \frac{\Gamma(\frac{\A+\C+s/\sqrt{\tau}}{2},\frac{\A-\C-s/\sqrt{\tau}}{2})}{2(\C+s/\sqrt{\tau}) \Gamma(-\C -s/\sqrt{\tau} )}
   \\ &\times \sum_{\{k\geq 0:\; \C+s/\sqrt{\tau}+2k<0 \}} e^{\tau(\C+s/\sqrt{\tau}+2k)^2/4}(\C+s/\sqrt{\tau}+2k)\tfrac{\left(\C+s/\sqrt{\tau},\tfrac{\A+\C+s/\sqrt{\tau}}{2}\right)_k}
   {k!\left(1+\tfrac{\C+s/\sqrt{\tau}-\A}{2}\right)_k} \\
   &= 2^{\A+\C+s/\sqrt{\tau}-1}  \frac{\Gamma(\frac{\A+\C+s/\sqrt{\tau}}{2},\frac{\A-\C-s/\sqrt{\tau}}{2})}{\Gamma(-\C -s/\sqrt{\tau} )}\\
   & \times \left(  e^{\tau(\C+s/\sqrt{\tau})^2/4} +
   \frac{(\C+s/\sqrt{\tau}+2)(\A+\C+s/\sqrt{\tau})}
   {2+ \C+s/\sqrt{\tau}-\A}e^{\tau(\C+s/\sqrt{\tau}+2)^2/4}+\dots \right).
\end{align*}
With $s<0$, we see that $e^{-s \C \sqrt{\tau}}I_\tau\to 0$ and that the leading term in the asymptotics comes from the first term in the sum $D_\tau$. We get
$$
\lim_{\tau\to\infty} L_\tau(s) =\lim_{\tau\to\infty}  e^{-s \C \sqrt{\tau}/2}\frac{D_\tau(s)}{D_\tau(0)}
=\lim_{\tau\to\infty}    \frac{ e^{\tau(\C+s/\sqrt{\tau})^2/4-s \C \sqrt{\tau}/2}}{ e^{\tau \C^2/4}} = e^{s^2/4}.$$
\end{proof}
\subsection*{Acknowledgements} We thank Guillaume Barraquand for an inspiring email  about his results and comments on the first draft of this paper, and to Yizao Wang for a discussion that led to Remark \ref{Rem:Yizao}. We thank the reviewer for the detailed report which helped us to clarify and improve the paper.
WB's research was partially supported by Simons Foundation  Award Number: 703475. AK's research was partially supported by
The Natural Sciences and Engineering Research Council of Canada.
 We are grateful to Yizao Wang for pointing out a mistake in \eqref{K2C} in the published version. This mistake has been corrected in the current arxiv version.

\end{document}